\documentclass[12pt,a4paper]{article}
\usepackage{mathtools,amsthm,amsmath,amsfonts,amssymb,tikz-cd,enumitem, graphicx,mathrsfs,bbm,appendix}

\tikzcdset{scale cd/.style={every label/.append style={scale=#1},
    cells={nodes={scale=#1}}}}
\usepackage{authblk}
\usepackage{url}
\usepackage{esvect}
\makeatletter
\newcommand{\address}[1]{\gdef\@address{#1}}
\newcommand{\email}[1]{\gdef\@email{\url{#1}}}
\newcommand{\@endstuff}{\par\vspace{\baselineskip}\noindent\small
\begin{tabular}{@{}l}\@address\\\textit{E-mail address:} \@email\end{tabular}}
\AtEndDocument{\@endstuff}
\makeatother
\usepackage[utf8]{inputenc}
\usepackage{CJKutf8}
\counterwithin{equation}{section}
\usepackage[pdfencoding=auto,psdextra,colorlinks=true,linkcolor=blue,citecolor=blue,urlcolor = orange]{hyperref}
\usepackage[margin = 1.25in]{geometry}
\usepackage{kantlipsum}
\allowdisplaybreaks
\newtheorem{theorem}{Theorem}[section]
\newtheorem{definition}[theorem]{Definition}

\newtheorem{lemma}[theorem]{Lemma}
\newtheorem{question}[theorem]{Question}
\newtheorem{remark}[theorem]{Remark}
\newtheorem{proposition}[theorem]{Proposition}
\newtheorem{corollary}[theorem]{Corollary}
\newtheorem{assumption}[theorem]{Assumption}

\newcommand{\citep}[1]{\cite{#1}}
\newcommand{\grad}{{\rm grad}}

\newcommand{\dvol}{\text{\normalfont dvol}}

\newcommand{\image}{\text{\normalfont Im}}

\newcommand{\supp}{\text{\normalfont supp}}

\newcommand{\cupc}{\mathcal{C}}
\newcommand{\coker}{\mathrm{coker}}
\begin{document}
%\onehalfspacing
\title{Instanton construction of the mapping cone Thom-Smale complex}
\author{Hao Zhuang}
\address{Beijing International Center for Mathematical Research, Peking University\\ Beijing, China}
\email{hzhuang@pku.edu.cn}
\date{\today}
\maketitle
\begin{abstract}
   The cup product structure on the topological side of the classical Thom-Smale complex leads to the topological side of the mapping cone Thom-Smale complex. Following the spirit of Witten's analytic Morse theory, we ask whether the mapping cone Thom-Smale complex has the analytic side. However, due to the missing cup product structure on the analytic side of the classical Thom-Smale complex, it seems that we can only have a hybrid analytic-topological construction of the mapping cone Thom-Smale complex. In this paper, we overcome the cup product issue and give the purely analytic construction of the mapping cone Thom-Smale complex. More precisely, for a Morse function with the transversality condition on a closed oriented Riemannian manifold, we construct an instanton cochain complex using the eigenspaces of the mapping cone Laplacian deformed by the Morse function and two parameters. One parameter gives the classical Witten deformation. The other parameter overcomes the cup product obstacle by suppressing the norm of the given differential form. As the main result, we prove that our instanton complex is cochain isomorphic to the topologically constructed mapping cone Thom-Smale complex, and therefore it is the purely analytic construction that we need. 
\end{abstract}
%\vspace{-2.5mm}
%\keywords{Mapping cone, Thom-Smale complex, Witten deformation}

%\noindent\subjectclass{58J20, 37D15 (Primary); 81Q60, 57R58 (Secondary)}
\tableofcontents
\section{Introduction}
The mapping cone de Rham complex induced by the wedge of a closed form has witnessed many interesting studies in recent years. In general, given a smooth closed $\ell$-form $\omega$ on a closed smooth manifold $M$, the mapping cone de Rham complex is\footnote{We allow $k = -1$ to include the case of $\ell = 0$, i.e., the case where $\omega$ is a constant. In this paper, $\Omega^k(M) = 0$ when $k<0$ or $k>m$.}
\begin{align}\label{wedging to mapping cone}
    d^\omega: \Omega^k(M)\oplus \Omega^{k-\ell+1}(M) &\to \Omega^{k+1}(M)\oplus \Omega^{k-\ell+2}(M)\ \ (-1\leqslant k\leqslant m+\ell-1) \nonumber\\
    (\alpha, \beta)& \mapsto (d\alpha+\omega\wedge\beta, (-1)^{\ell-1}d\beta).
\end{align}
By the interaction between analytic and topological tools, we see that extra information arises from $\omega\wedge$ when we transfer current geometry and topology from the original de Rham cochain complex to (\ref{wedging to mapping cone}). A well-known result is that when $\omega$ is the $n$-th power of a symplectic form, the mapping cone de Rham complex computes the $n$-filtered cohomology \cite{tanaka_tseng_2018, tty3rd, tty1st, tty2nd} of $M$, which is given by the Lefschetz decomposition of $\Omega^k(M)$ and depends on the symplectic structure.  

In this paper, we study the mapping cone Morse theory started by Clausen, Tang, and Tseng \cite{clausen_tang_tseng_2024_mappingconemorsetheory, tangtsengclausensymplecticwitten}. For any Morse-Smale pair $(f,g)$ of a function $f$ and a metric $g$ on $M$, Clausen, Tang, and Tseng constructed the mapping cone Thom-Smale cochain complex in the topological way \cite{clausen_tang_tseng_2024_mappingconemorsetheory} as follows. We let $C^k(f,g)$ be the linear space spanned by all index-$k$ critical points of $f$, let 
\begin{align}
    \partial: C^k(f,g)\to C^{k+1}(f,g)
\end{align}
be the associated classical Thom-Smale cochain complex, and let 
\begin{align}
    \mathcal{C}(\omega): C^k(f,g)\to C^{k+\ell}(f,g)
\end{align}
be the cup product of $\omega$ on the Thom-Smale cochain complex. Then, the map
\begin{align}\label{introduce the mapping cone thom-smale cochain complex}
\begin{split}
    \partial^\omega: C^k(f, g)\oplus C^{k-\ell+1}(f, g) &\to C^{k+1}(f, g)\oplus C^{k-\ell+2}(f, g)\\
    \begin{bmatrix}
            a\\
            b
        \end{bmatrix}& \mapsto \begin{bmatrix}
            \partial & \mathcal{C}(\omega)\\
            0 & (-1)^{\ell-1}\partial
        \end{bmatrix}\begin{bmatrix}
            a\\
            b
        \end{bmatrix}
\end{split}
\end{align}
defines the mapping cone Thom-Smale cochain complex. By \cite[Theorem 1.2]{clausen_tang_tseng_2024_mappingconemorsetheory}, the complex (\ref{introduce the mapping cone thom-smale cochain complex}) is quasi-isomorphic to (\ref{wedging to mapping cone}). 

Meanwhile in \cite{tangtsengclausensymplecticwitten}, Clausen, Tang, and Tseng proposed another construction using the deformed de Rham exterior derivative and an ``$\omega\wedge$ type map'' on the classical Witten instanton complex (See \cite{Bismut_Zhang_1992_theorem_Cheeger_Muller, Bismut1994, Helffer01011985, witten,  wittendeformationweipingzhang} for the Witten instanton complex of a Morse-Smale pair). It is here that we truly encounter the serious issue brought by the wedge $\omega\wedge$. 
We let $d_T = d+Tdf$, and $F_T^k(f,g)$ be the sum of eigenspaces of the Hodge-Witten Laplacian 
\begin{align}
    (d_T+d_T^*)^2: \Omega^k(M)\to\Omega^k(M)
\end{align}
associated with eigenvalues in $[0,1]$. Recall the classical Witten instanton cochain complex 
\begin{align}
    d_T: F_T^k(f,g) \to F_T^{k+1}(f,g).
\end{align}
When $T$ is sufficiently large, there is a cochain isomorphism 
\begin{align}\label{auxiliary isomorphism on the chain level}
    \Phi_T: (F_T^\bullet(f,g), d_T)\cong (C^\bullet(f,g), \partial).
\end{align}
 We could expect the map 
\begin{align}\label{tempting attempt}
   \begin{split}
       F_T^k(f,g)\oplus F_T^{k-\ell+1}(f,g)&\to F_T^{k+1}(f,g)\oplus F_T^{k-\ell+2}(f,g)\\
    \begin{bmatrix}
        \alpha\\
        \beta
    \end{bmatrix} &\mapsto \begin{bmatrix}
        d_T & \omega\\
        0 & (-1)^{\ell-1}d_T
    \end{bmatrix}\begin{bmatrix}
        \alpha\\
        \beta
    \end{bmatrix}
   \end{split}
\end{align}
to give a purely analytic instanton construction of the mapping cone Thom-Smale complex. However, the wedge $\omega\wedge$ does not preserve the space $F_T^{k}(f,g)$, and therefore (\ref{tempting attempt}) does not define a cochain complex. The construction in \cite{tangtsengclausensymplecticwitten} is replacing $\omega$ in (\ref{tempting attempt}) by the ``$\omega\wedge$ type map'' 
\begin{align}\label{the map is not trivial}
    \Phi_T^{-1}\circ \mathcal{C}(\omega)\circ \Phi_T.
\end{align} 
This construction is hybrid analytic-topological. Comparing with the patterns in the previous analytic studies \cite{dai_yan_2021, Helffer01011985, lu2017thom, witten, wittendeformationweipingzhang} on the Thom-Smale complex, and also for further discussions on topics like the mapping cone version of the analytic torsion \cite{Bismut_Zhang_1992_theorem_Cheeger_Muller, Bismut1994}, we are motivated to study this question: 
 
\begin{question}
\normalfont
    For a Morse-Smale pair on a smooth closed oriented manifold, how do we give a more canonical purely analytic instanton construction of its mapping cone Thom-Smale complex, i.e., only using the eigenspaces of a ``Laplacian'' without the auxiliary map $\Phi_T$?
\end{question}

In this paper, we give an answer to this question. We assume:  
\begin{assumption}\label{basic assumption}
\normalfont
	$M$ is an $m$-dimensional closed oriented manifold, $\omega$ is a closed smooth $\ell$-form on $M$, $f$ is a Morse function on $M$, and $g$ is a Riemannian metric on $M$ such that $(f,g)$ is a Morse-Smale pair.
\end{assumption}

Before studying the mapping cone situation, we quickly review the classical Thom-Smale cohomology that we have mentioned and set up notations. 

First, we let $Z(df)$ be the set of critical points of $f$, $Z^k(df)$ be the set of those with Morse index $k$, and $C^k(f, g)$ be the real vector space spanned by $Z^k(df)$. 

 Second, each $p\in Z(df)$ has the unstable manifold $\mathcal{U}(p)$ and the stable manifold $\mathcal{S}(p)$ with respect to the negative gradient of $f$ under $g$. 
 For any $p,q\in Z(df)$, $\mathcal{U}(q)$ intersects $\mathcal{S}(p)$ transversely. We let
 \begin{align}
     \mathcal{M}(q,p) = \mathcal{U}(q)\cap \mathcal{S}(p), 
 \end{align}
and let $\overline{\mathcal{M}(q,p)}$ be the closure of $\mathcal{M}(q,p)$. For each $q\in Z(df)$, we fix an orientation on $\mathcal{U}(q)$. Then, we obtain the orientation on $\mathcal{M}(q,p)$ and that on $\mathcal{M}(q,p)/\mathbb{R}$ according to the rule \cite[Section 2.2, (3)]{hutchingslecturenotes}.

Third, we let $\grad(f)$ be the gradient of $f$ with respect to $g$, and 
\begin{align}\label{classical thom smale cochain complex}
    \partial: C^k(f, g)\to C^{k+1}(f, g)
\end{align}
be the coboundary map of the classical Thom-Smale cochain complex. This $\partial$ is defined by counting (with signs) the flow lines of $-\grad(f)$ between critical points. 

Fourth, following \cite[(2.1)]{Austin1995}, \cite[Theorem 1]{Viterbo1995}, and \cite[(6)]{clausen_tang_tseng_2024_mappingconemorsetheory}, we define
    \begin{align}\label{cup product map on chain elements}
   \begin{split}
   	\mathcal{C}(\omega): C^k(f, g)&\to C^{k+\ell}(f, g)\\
    p\in Z^k(df)&\mapsto \sum_{q\in Z^{k+\ell}(df)}\left(\int_{\overline{\mathcal{M}(q,\hspace{+0.375mm} p)}} \omega\right)q.
   	\end{split}
   \end{align}
   This $\mathcal{C}(\omega)$ is the cup product map given by $\omega$. 
Then, we have the mapping cone Thom-Smale complex \cite[Definition 1.1]{clausen_tang_tseng_2024_mappingconemorsetheory}: 
\begin{definition}\normalfont
    We call the cochain complex given by
\begin{align}\label{topological definition}
\begin{split}
    \partial^\omega: C^k(f, g)\oplus C^{k-\ell+1}(f, g) &\to C^{k+1}(f, g)\oplus C^{k-\ell+2}(f, g)\\
    \begin{bmatrix}
            a\\
            b
        \end{bmatrix}& \mapsto \begin{bmatrix}
            \partial & \mathcal{C}(\omega)\\
            0 & (-1)^{\ell-1}\partial
        \end{bmatrix}\begin{bmatrix}
            a\\
            b
        \end{bmatrix}
\end{split}
\end{align}
 the mapping cone Thom-Smale complex of $(f, g)$. It satisfies $\partial^\omega\circ\partial^\omega = 0.$
\end{definition}

  Next, we construct our purely analytic side. 
  
  For any $S>0$ and $T\geqslant 0$, we see that the map (cf. \cite[Remark A.3]{tanaka_tseng_2018} and \cite[Proposition 5.3]{wittendeformationweipingzhang}) 
   \begin{align}\label{perturbed deformed mapping cone de rham complex}
   \begin{split}
       d^\omega_{ST}: \Omega^k(M)\oplus\Omega^{k-\ell+1}(M)&\to\Omega^{k+1}(M)\oplus\Omega^{k-\ell+2}(M)\\
    \begin{bmatrix}
        \alpha\\
        \beta
    \end{bmatrix} &\mapsto \begin{bmatrix}
        d+Tdf & S^{-1}\omega\\
        0 & (-1)^{\ell-1}(d+Tdf)
    \end{bmatrix}\begin{bmatrix}
        \alpha\\
        \beta
    \end{bmatrix}
   \end{split}
\end{align}
satisfies $d_{ST}^\omega\circ d_{ST}^\omega = 0$ and thus defines a cochain complex. It is a Witten type deformation of (\ref{wedging to mapping cone}). We will explain the necessity of the parameter $S$ later. 

\begin{remark}\normalfont
    By Proposition \ref{routinely check the deformation does not change the cohomology}, (\ref{perturbed deformed mapping cone de rham complex}) and (\ref{wedging to mapping cone}) are cochain isomorphic to each other. Then, the topological result 
\cite[Theorem 1.2]{clausen_tang_tseng_2024_mappingconemorsetheory} shows that (\ref{perturbed deformed mapping cone de rham complex}) is quasi-isomorphic to the mapping cone Thom-Smale complex (\ref{topological definition}). However, this quasi-isomorphism is not the final goal of the analytic study. 
\end{remark}

The metric $g$ induces an $L^2$ norm $\|\cdot\|$ on $\Omega^k(M)$. This norm and its associated inner product naturally extends to the direct sum 
\begin{align}
    \bigoplus_{k = -1}^{m+\ell-1}\left(\Omega^k(M)\oplus\Omega^{k-\ell+1}(M)\right) \cong \left(\bigoplus_{k = -1}^{m+\ell-1}\Omega^k(M)\right)\oplus\left(\bigoplus_{k = -1}^{m+\ell-1}\Omega^{k-\ell+1}(M)\right). 
\end{align}
Then, we have the formal adjoint ${d_{ST}^\omega}^*$ of $d_{ST}^\omega$. 
Our purely analytic instanton complex of $(f,g,\omega)$ is defined as follows:
   
   \begin{definition}\label{define the mapping cone witten instanton}
   \normalfont
   For $S>0$ and $T\geqslant 0$, we let $\mathbb{D}_{ST} = d^\omega_{ST} + {d_{ST}^\omega}^*$, and 
   	\begin{align}\label{definition of fstkfgomega}
	F_{ST}^k(f,g,\omega)=\bigoplus_{0\leqslant\lambda\leqslant 1}\left\{\mathbf{w}\in\Omega^k(M)\oplus\Omega^{k-\ell+1}(M): \mathbb{D}_{ST}^2\mathbf{w} = \lambda\mathbf{w}\right\}.
    \end{align}
    We call the cochain complex given by 
    \begin{align}\label{analytic definition}
    d_{ST}^\omega: F_{ST}^k(f,g,\omega)\to F_{ST}^{k+1}(f,g,\omega)\ \ (-1\leqslant k\leqslant m+\ell-1)
    \end{align}
    the instanton cochain complex of $(f,g,\omega)$ with parameters $S$ and $T$. 
   \end{definition}
   
\begin{remark}
    \normalfont Comparing with \cite[(2.9)]{tangtsengclausensymplecticwitten}, Definition \ref{define the mapping cone witten instanton} is purely analytic. It does not use the auxiliary map (\ref{auxiliary isomorphism on the chain level}). 
\end{remark}

\begin{remark}
\normalfont
    Even when $S=1$, the space (\ref{definition of fstkfgomega}) is not equal to $F_T^{k}(f,g)\oplus F_T^{k-\ell+1}(f,g)$ in (\ref{tempting attempt}). In fact, later our main result reveals that if $T$ is sufficiently large, and $S$ is exponentially larger than $T$, we have $F_T^{k}(f,g)\oplus F_T^{k-\ell+1}(f,g)\cong F_{ST}^k(f,g,\omega)$. 
\end{remark}

   By \cite[Proposition 5.1]{Helffer01011985}, we can perturb our Morse-Smale pair $(f,g)$ so that: 
   \begin{enumerate}[label = (a\arabic*)]
       \item The critical set $Z(df)$ and the index of each $p\in Z(df)$ are unchanged. \label{perturbation 1}
       \item For each $p\in Z^k(df)$, there is a local chart $U = (x_1,\cdots,x_m)$ centered at $p$ such that these $U$'s are disjoint, and \label{perturbation 2}
       \begin{align}
       f|_U =\ & f(p)-\dfrac{1}{2}(x_1^2+\cdots+x_k^2) + \dfrac{1}{2}(x_{k+1}^2+\cdots+x_m^2), \label{morse lemma} \\
       g|_U =\ & dx_1^2+\cdots + dx_m^2. \label{locally euclidean by morse lemma}
   \end{align}
   \item $(f,g)$ is still a Morse-Smale pair. \label{perturbation 3}
   \end{enumerate}
   
   As the main result, more than a quasi-isomorphism, we give a cochain isomorphism: 
   \begin{theorem}\label{main theorem}
   	We perturb $(f, g)$ subject to {\normalfont \ref{perturbation 1}, \ref{perturbation 2}, and \ref{perturbation 3}}. Then, we have constants $C_0>1$ and $T_0>0$ such that when $S > e^{C_0T} > e^{C_0T_0}$, there is a cochain isomorphism between the instanton complex {\normalfont(\ref{analytic definition})} of $(f,g,\omega)$
    and the mapping cone Thom-Smale complex {\normalfont(\ref{topological definition})} of $(f,g)$.
   \end{theorem}
   
Theorem \ref{main theorem} shows that our Definition \ref{define the mapping cone witten instanton} gives a purely analytic instanton construction of the mapping cone Thom-Smale complex. In this construction, the operator $\mathbb{D}_{ST}^2$ is the Laplacian that we need. 

The usage of the parameter $T$ originates from the Witten deformation technique \cite{Bismut_Zhang_1992_theorem_Cheeger_Muller, Bismut1994, Helffer01011985, witten, wittendeformationweipingzhang}. We now explain the necessity of the other parameter $S$, which is forced by the wedge $\omega\wedge$. As we know, $\omega\wedge$ works on $\Omega^k(M)$, so at the first attempt, we start with 
\begin{align}
   \begin{split}
       d_T^\omega: \Omega^k(M)\oplus\Omega^{k-\ell+1}(M)&\to\Omega^{k+1}(M)\oplus\Omega^{k-\ell+2}(M)\\
    \begin{bmatrix}
        \alpha\\
        \beta
    \end{bmatrix} &\mapsto \begin{bmatrix}
        d + Tdf & \omega\\
        0 & (-1)^{\ell-1}(d + Tdf)
    \end{bmatrix}\begin{bmatrix}
        \alpha\\
        \beta
    \end{bmatrix}
   \end{split}
\end{align}
and its associated Laplacian $(d_T^\omega + {d_T^\omega}^*)^2$. We let $F_T^k(f,g,\omega)$ be the sum of eigenspaces of $(d_T^\omega + {d_T^\omega}^*)^2$ associated with eigenvalues in $[0,1]$ and obtain the cochain complex
\begin{align}\label{not the final answer of instanton construction}
    d_T^\omega: F_T^k(f,g,\omega)\to F_T^{k+1}(f,g,\omega).
\end{align}
Like in \cite[(5.18), (6.14)]{wittendeformationweipingzhang}, we want to obtain a suitable space $\mathbb{E}_T$ which approximates $\ker((d_T^\omega + {d_T^\omega}^*)^2|_U)$ as precisely as possible in the coordinate chart $U$ of each critical point of $f$. Aiming at finding the cochain isomorphism between (\ref{not the final answer of instanton construction}) and  (\ref{topological definition}), we wish to show that the projection of $\mathbb{E}_T$ onto $F_T^k(f,g,\omega)$ is close to $\mathbb{E}_T$ itself. In other words, we need a space $\mathbb{E}_T$ to play the role of a model for eigenforms associated with small eigenvalues of the Laplacian in the mapping cone setting. However, no matter what $\omega$ is, the unknown expression of $\omega$ forces only one choice of the space $\mathbb{E}_T$ (See (\ref{the definition of E_T space here}) and compare with \cite[Remark 3.5]{zhuang2026symplectic_char}). To make things worse, the projection of $\mathbb{E}_T$ deviates too much from $\mathbb{E}_T$ due to the norm of $\omega$. This phenomenon is a big difference between the classical setting and the mapping cone setting. In general, we do not have a cochain isomorphism between (\ref{not the final answer of instanton construction}) and  (\ref{topological definition}).

Fortunately, although we only have one choice of $\mathbb{E}_T$, we find that if the norm of $\omega$ is sufficiently small, the projection of $\mathbb{E}_T$ onto $F_T^k(f,g,\omega)$ is close to $\mathbb{E}_T$. Thus, we suppress the norm of $\omega$ by introducing the parameter $S$ in (\ref{perturbed deformed mapping cone de rham complex}) without changing the cohomology. Then, we replace $F_T^k(f,g,\omega)$ in (\ref{not the final answer of instanton construction}) by $F_{ST}^k(f,g,\omega)$ in (\ref{definition of fstkfgomega}). Once we let $S$ be much larger than $T$, the projection of $\mathbb{E}_T$ onto $F_{ST}^k(f,g,\omega)$ is close to $\mathbb{E}_T$. The details are in Proposition \ref{refinement of the estimates of proj vs non-proj}. The usage of $S$ reveals the influence of the cup product issue when we go from the classical Morse case to the mapping cone Morse case, echoing that $\mathcal{C}(\omega)$ commutes with the map (\ref{auxiliary isomorphism on the chain level}) on the cohomological level instead of on the cochain level \cite[Theorem 3.11]{Austin1995}, and that (\ref{the map is not trivial}) is not equal to $\omega\wedge$. 

\begin{remark}\normalfont
    If we only need an isomorphism between vector spaces instead of a cochain isomorphism between cochain complexes, the closeness between $\mathbb{E}_T$ and its projection to $F_{ST}^k(f,g,\omega)$ given in Proposition \ref{refinement of the estimates of proj vs non-proj} is not necessary. See Proposition \ref{a preliminary prerequisite for the final isomorphism but this is just the very rough one} for details.
\end{remark}

\begin{remark}\normalfont
    By \cite[Theorem 7.1]{tanaka_tseng_2018}, we could intuitively view the part of $\omega$ and $\omega^*$ in $(d_T^\omega + {d_T^\omega}^*)^2$ as the ``$d+d^*$'' on the ``critical fiber'' of the pullback of $f$ on a sphere bundle over $M$. This seems to provide a degenerate analytic Morse approach only using one parameter $T$. However, even for an $\omega$ that admits such a sphere bundle
    \begin{align}
        \pi: E\to M, 
    \end{align}
    both $\omega$ and the $1$-form $\theta$ satisfying $d\theta = \pi^*\omega$ are not fully restricted on the fiber direction, meaning that $\omega$ and $\omega^*$ do not exactly give the ``$d+d^*$'' that we want. 
\end{remark}

 As an application of Theorem \ref{main theorem}, we have the following precise analytic expression of the mapping cone Morse inequalities \cite[Theorems 1.3, 3.4]{clausen_tang_tseng_2024_mappingconemorsetheory}. 
\begin{corollary}\label{morse equalities}
	Let $\mu_k = |Z^k(df)|$, $b^\omega_k$ be the dimension of the $k$-th cohomology group of the complex {\normalfont(\ref{wedging to mapping cone})}, and $R_k$ be the rank of the map {\normalfont(\ref{analytic definition})}. 
	Then, we have
	\begin{align}
		R_k+\sum_{j = -1}^k (-1)^{k-j} b_j^\omega = \sum_{j = -1}^k (-1)^{k-j} (\mu_j + \mu_{j-\ell+1})
	\end{align}
	for $-1\leqslant k\leqslant m+\ell-1$ under the conditions in Theorem \ref{main theorem}. 
\end{corollary}

By Corollary \ref{morse equalities} and Theorem \ref{main theorem}, we study $R_k$ and reproduce the topological mapping cone Morse inequalities \cite[Theorems 1.3, 3.4]{clausen_tang_tseng_2024_mappingconemorsetheory} in a simple way: 
\begin{corollary}[Clausen-Tang-Tseng \cite{clausen_tang_tseng_2024_mappingconemorsetheory}, 2026]\label{reproducing inequalities here by alternating sums}
	We follow the notations in Corollary \ref{morse equalities}. Let $v_k$ be the rank of the cup product map {\normalfont(\ref{cup product map on chain elements})}. Then, we have 
	\begin{align}\label{morse inequalities from the 2026 paper}
		\sum_{j = -1}^k (-1)^{k-j} b_j^\omega \leqslant \sum_{j = -1}^k (-1)^{k-j} (\mu_j - v_{j-\ell} + \mu_{j-\ell+1} - v_{j-\ell+1})
	\end{align}
	for $-1\leqslant k\leqslant m+\ell-1$. 
\end{corollary}

Next, we let $H^k_{ST}(f,g,\omega)$ be the $k$-th cohomology of (\ref{analytic definition}), and $H^k(f,g)$ be the $k$-th cohomology group of the classical Thom-Smale complex (\ref{classical thom smale cochain complex}). 
\begin{corollary}\label{decomposition corollary}
The cohomology group $H^k_{ST}(f, g, \omega)$ is isomorphic to 
{\small{\begin{align}\label{decomposition expression in introduction}
\hspace{-2mm}
    \coker\left(\cupc(\omega): H^{k-\ell}(f,g)\to H^{k}(f,g)\right)\oplus
    \ker\left(\cupc(\omega): H^{k-\ell+1}(f,g)\to H^{k+1}(f,g)\right)
\end{align}
}}
\hspace{-1.5mm}under the conditions in Theorem \ref{main theorem}. 
\end{corollary}
Although (\ref{decomposition expression in introduction}) is directly implied by \cite[Theorem 1.2]{clausen_tang_tseng_2024_mappingconemorsetheory} together with the mapping cone algebraic property of (\ref{topological definition}), we will give a detailed proof using Theorem \ref{main theorem} and diagram chasing like \cite[(2.12)]{tangtsengclausensymplecticwitten}. Unlike \cite[(2.9)]{tangtsengclausensymplecticwitten}, our instanton complex is not explicitly expressed as a mapping cone complex, but it still has the property (\ref{decomposition expression in introduction}).
An important application of (\ref{decomposition expression in introduction}) is to give an exact sequence approach to prove the mapping cone Morse inequalities \cite[Theorems 1.3, 3.4]{clausen_tang_tseng_2024_mappingconemorsetheory}. This approach provides a more natural explanation to the appearance of the rank $v_k$ of $\cupc(\omega)|_{C^k(f,g)}$ in (\ref{morse inequalities from the 2026 paper}). We refer readers to \cite[Section 3]{clausen_tang_tseng_2024_mappingconemorsetheory} for this elegant approach. 

Similar to the role of the classical Witten instanton complex in \cite{Bismut_Zhang_1992_theorem_Cheeger_Muller, Bismut1994}, a further application of our work would be studying the mapping cone version of the Ray-Singer analytic torsion \cite{ray_singer_1971r, ray_singer_1973analytic}, the associated Cheeger-M\"uller theorem \cite{cheeger1979analytic, muller1978analytic}, and the associated Bismut-Zhang theorem \cite{Bismut_Zhang_1992_theorem_Cheeger_Muller, Bismut1994}. These could be the topics where we need the purely analytic instanton construction the most. Also, we suggest studying the group action case, for example, when $\omega$ is symplectic and compatible with a group action \cite{daSilva2008}. The function then admits assumptions like those in the classical degenerate Morse cases \cite{bao2026equivariantmorsebottcohomologystabilization, wenlumorsebottineq, zhuang2025analytictopologicalrealizationsinvariant} under group actions. The analysis that we have used could still work due to certain symmetry. Here, we do not discuss the details.

This paper is organized in this order: In Section \ref{section of Review of the topological side}, we review the topological construction of the mapping cone Thom-Smale complex. In Section \ref{section of Mapping cone Hodge theory}, we review the Hodge theory for the mapping cone de Rham complex and explain the instanton complex. In Section \ref{section of Harmonic oscillators around critical points}, we present the eigenvalue behaviors of the mapping cone Laplacian when $S$ and $T$ are sufficiently large. In Section \ref{section of Isomorphism between cochain complexes}, we prove Theorem \ref{main theorem}. In Section \ref{section of Decomposition of instanton cohomology}, we prove Corollaries \ref{morse equalities}-\ref{decomposition corollary}. 

\section{Review of the topological side}\label{section of Review of the topological side}
In this section, we review the classical Thom-Smale complex \cite{Austin1995, hutchingslecturenotes, Viterbo1995} and the topological side \cite{clausen_tang_tseng_2024_mappingconemorsetheory} of the mapping cone Thom-Smale complex.

For any $p, q\in Z(df)$, we let
\begin{align}
    \widetilde{\mathcal{M}}(q, p) = \mathcal{M}(q,p)/\mathbb{R}.
\end{align}
For each $\mathcal{U}(p)$, we fix its orientation $[\mathcal{U}(p)]$. According to \cite[Section 2.2, (3)]{hutchingslecturenotes}, the orientation of $\widetilde{\mathcal{M}}(q, p)$ and that of $\mathcal{M}(q, p)$ are chosen by 
\begin{align}
    [\mathcal{U}(q)] = [\mathcal{M}(q,p)][\mathcal{U}(p)] = [\widetilde{\mathcal{M}}(q,p)][-\grad(f)][\mathcal{U}(p)].
\end{align} 
Using the closure
$\overline{\widetilde{\mathcal{M}}(q, p)}$ (resp. $\overline{\mathcal{M}(q, p)}$) of $\widetilde{\mathcal{M}}(q, p)$ (resp. $\mathcal{M}(q, p)$), we have the coboundary map
\begin{align}\label{classical Thom-Smale complex}
    \begin{split}
        \partial: C^k(f, g)&\to C^{k+1}(f, g)\\
        p&\mapsto\sum_{q\in Z^{k+1}(df)} \left(\int_{\overline{\widetilde{\mathcal{M}}(q,\hspace{+0.375mm}p)}} 1 \right)q
    \end{split}
\end{align}
and the cup product map
\begin{align}
   \begin{split}
   	\mathcal{C}(\omega): C^k(f, g)&\to C^{k+\ell}(f, g)\\
    p\in Z^k(df)&\mapsto \sum_{q\in Z^{k+\ell}(df)}\left(\int_{\overline{\mathcal{M}(q,\hspace{+0.375mm} p)}} \omega\right)q.
   	\end{split}
   \end{align}
The $\partial$ satisfies $\partial^2 = 0$.
\begin{definition}\normalfont
    The cochain complex (\ref{classical Thom-Smale complex}) is the classical Thom-Smale cochain complex associated with $(f,g)$. The associated cohomology groups are called the Thom-Smale cohomology groups of $M$. 
\end{definition}

In addition, since $\omega$ is closed, by \cite[(7)]{clausen_tang_tseng_2024_mappingconemorsetheory}, we have:
\begin{proposition}
The cup product map $\mathcal{C}(\omega)$ satisfies
    \begin{align}\label{cup product map simplified leibniz rule}
    \partial\mathcal{C}(\omega) = (-1)^{\ell}\mathcal{C}(\omega)\partial,
\end{align}
i.e., it anti-commutes with $\partial$. 
\end{proposition}
The sign $(-1)^{\ell}$ refines \cite[(2.2)]{Austin1995} and \cite[Lemma 3]{Viterbo1995}. See \cite[Appendix A]{clausen_tang_tseng_2024_mappingconemorsetheory} for the proof. 

\begin{remark}\normalfont 
    With the orientation convention \cite[Section 2.2, (3)]{hutchingslecturenotes}, 
    if we replace $\omega$ by any smooth $\ell$-form $\alpha$ on $M$, then 
    \begin{align}\label{the current convention of orientation leads to the number}
        \partial(\mathcal{C}(\alpha)) = \mathcal{C}(d\alpha) + (-1)^\ell\mathcal{C}(\alpha)\partial.
    \end{align}
    Originally in \cite{clausen_tang_tseng_2024_mappingconemorsetheory}, with the orientation convention \cite[(54)]{clausen_tang_tseng_2024_mappingconemorsetheory}, the direction of $-\grad f$ is replaced by the direction of $\grad f$. Thus, under \cite[(54)]{clausen_tang_tseng_2024_mappingconemorsetheory}, \cite[Lemma A.1]{clausen_tang_tseng_2024_mappingconemorsetheory} gives
    \begin{align}
        \partial(\mathcal{C}(\alpha)) = -\mathcal{C}(d\alpha) + (-1)^\ell\mathcal{C}(\alpha)\partial. \tag{\ref*{the current convention of orientation leads to the number}'}
    \end{align}
    We adopt \cite[Section 2.2, (3)]{hutchingslecturenotes} since (\ref{the current convention of orientation leads to the number}) is closer to the Leibniz rule of $d$.
\end{remark}
The classical Thom-Smale cochain complex computes the Betti numbers of $M$. For each $p\in Z(df)$, we let $\overline{\mathcal{U}(p)}$ be the closure of $\mathcal{U}(p)$. Then, we have a quasi-isomorphism $\Phi$ between (\ref{classical Thom-Smale complex}) and the de Rham complex of $M$:
\begin{align}\label{original quasi isomorphism between classical thom-smale and de rham}
\begin{split}
    \Phi: \Omega^k(M) &\to C^k(f, g)\\
             \alpha   &\mapsto \sum_{p\in Z^k(df)}\left(\int_{\overline{\mathcal{U}(p)}}\alpha\right)p.
\end{split}
\end{align}
\begin{proposition}
    The map $\Phi$ is a cochain map satisfying 
    \begin{align}
        \Phi\circ d = \partial \circ \Phi.
    \end{align}
    Moreover, it induces an isomorphism on cohomology groups. 
\end{proposition}

By \cite[Theorem 2.13]{Austin1995} and \cite[Theorem 1]{Viterbo1995}, the cup product map is compatible with the quasi-isomorphism on the level of cohomology.
\begin{proposition}\label{cup and Phi compatible on the level of cohomology}
    For any closed $\beta\in\Omega^k(M)$, $\mathcal{C}(\omega)\Phi(\beta)$ and $\Phi(\omega\wedge\beta)$ represents the same Thom-Smale cohomology class.  
\end{proposition}

\begin{remark}\normalfont 
    We need to adjust $\omega$ and $\beta$ by exact forms so that $\omega$ and $\beta$ are localized for integrations \cite[Lemma 4]{Viterbo1995}. Thus, the compatibility between $\mathcal{C}(\omega)$ and $\Phi$ holds only on cohomology classes. 
\end{remark}

Now, we present the mapping cone situation. Given the closed $\ell$-form $\omega$, we have the mapping cone de Rham cochain complex
   \begin{align}\label{mapping cone de Rham complex}
   \begin{split}
       d^\omega: \Omega^k(M)\oplus\Omega^{k-\ell+1}(M)&\to\Omega^{k+1}(M)\oplus\Omega^{k-\ell+2}(M)\ \ (-1\leqslant k\leqslant m+\ell-1)\\
    \begin{bmatrix}
        \alpha\\
        \beta
    \end{bmatrix} &\mapsto \begin{bmatrix}
        d & \omega\\
        0 & (-1)^{\ell-1}d
    \end{bmatrix}\begin{bmatrix}
        \alpha\\
        \beta
    \end{bmatrix} = \begin{bmatrix}
        d\alpha + \omega\wedge\beta\\
        (-1)^{\ell-1}d\beta
    \end{bmatrix}.
   \end{split}
\end{align}
The map $d^\omega$ satisfies $d^\omega\circ d^\omega = 0$ since $\omega$ is closed. 
\begin{definition}\normalfont
    We call (\ref{mapping cone de Rham complex}) the mapping cone de Rham cochain complex of $(M,\omega)$. 
\end{definition}

\begin{remark}\normalfont
   Let $b^\omega_k$ be the dimension of the $k$-th cohomology group of (\ref{mapping cone de Rham complex}). The values of $b_k^\omega$'s depend on the choices of $\omega$. When $\omega$ is some power of a symplectic form, the mapping cone de Rham complex computes the filtered cohomology groups of $(M,\omega)$. When the symplectic form is furthermore integral, the mapping cone de Rham complex computes the Betti numbers of a sphere bundle over $M$. 
\end{remark}

Following \cite[Definition 1.1]{clausen_tang_tseng_2024_mappingconemorsetheory}, we have the mapping cone version of the Thom-Smale complex. Let $\partial^\omega$ be the map
\begin{align}\label{thom smale mapping cone version complex review}
\begin{split}
    \partial^\omega: C^k(f, g)\oplus C^{k-\ell+1}(f, g) &\to C^{k+1}(f, g)\oplus C^{k-\ell+2}(f, g)\ \ (-1\leqslant k\leqslant m+\ell-1)\\
    \begin{bmatrix}
            a\\
            b
        \end{bmatrix}& \mapsto \begin{bmatrix}
            \partial & \mathcal{C}(\omega)\\
            0 & (-1)^{\ell-1}\partial
        \end{bmatrix}\begin{bmatrix}
            a\\
            b
        \end{bmatrix} = \begin{bmatrix}
            \partial a + \mathcal{C}(\omega)b\\
            (-1)^{\ell-1}\partial b
        \end{bmatrix}.
\end{split}
\end{align}
The $\partial^\omega$ satisfies $\partial^\omega\circ\partial^\omega = 0$ because of (\ref{cup product map simplified leibniz rule}). 

\begin{definition}\normalfont
    We call (\ref{thom smale mapping cone version complex review}) the mapping cone Thom-Smale cochain complex of $(M,\omega)$ associated with $(f,g)$. 
\end{definition}

There is a quasi-isomorphism \cite[Definition 2.1]{clausen_tang_tseng_2024_mappingconemorsetheory} between the mapping cone de Rham complex and the mapping cone Thom-Smale complex. To begin, we assign 
$C^k(f, g)$ an inner product $\langle\cdot,\cdot\rangle$ by letting 
\begin{align}\label{inner product on Ck morse}
    \langle p,q\rangle = \begin{cases}
        1\ \text{when $p = q$}\\
        0\ \text{when $p \neq q$}
    \end{cases}
\end{align}
for $p,q\in Z^k(df)$. Then, for any 
\begin{align}
a\in\image(\partial: C^k(f, g)\to C^{k+1}(f, g)),
\end{align}
there is a unique $b\in C^k(f, g)$ such that $\partial b = a$ and that $b$ is orthogonal to 
\begin{align}
    \ker(\partial: C^k(f, g)\to C^{k+1}(f, g))
\end{align}
with respect to the inner product (\ref{inner product on Ck morse}). Then, we write $\delta(a) = b$. 

Given any $\beta\in \Omega^{k-\ell+1}(M)$, by the Hodge decomposition theorem \cite[Theorem 6.8]{warner2013foundations}, we have a unique way to write 
\begin{align}
    \beta = \beta_0 + dd^*\beta_1 + d^*d\beta_1\ \ (\beta_0, \beta_1\in\Omega^{k-\ell+1}(M))
\end{align}
under the $L^2$ norm (induced by $g$) on differential forms. Here, $d^*$ is the formal adjoint of $d$, $\beta_0$ is harmonic, and $\beta_0$ and $\beta_1$ are orthogonal to each other.  
By Proposition \ref{cup and Phi compatible on the level of cohomology}, 
\begin{align}
    \Phi(\omega\wedge\beta_0)-\mathcal{C}(\omega)\Phi(\beta_0)\in \image(\partial: C^k(f, g)\to C^{k+1}(f, g))
\end{align}
and thus we can define $\delta\left(\mathcal{C}(\omega)\Phi(\beta)-\Phi(\omega\wedge\beta)\right)$. Let $K(\omega)$ be the map
\begin{align}\label{the map K compatible with cup product}
    \begin{split}
        K(\omega): \Omega^{k-\ell+1}(M) &\to C^k(f, g)\\
        \beta &\mapsto  (-1)^\ell\left(\Phi(\omega\wedge d^*\beta_1)-\mathcal{C}(\omega)\Phi(d^*\beta_1)\right)+\delta\left(\Phi(\omega\wedge\beta_0)-\mathcal{C}(\omega)\Phi(\beta_0)\right).
    \end{split}
\end{align}
Using (\ref{original quasi isomorphism between classical thom-smale and de rham}) and (\ref{the map K compatible with cup product}), we obtain the quasi-isomorphism
\begin{align}\label{mapping cone morse and mapping cone de rham quasi-isom expression}
    \begin{split}
        \Phi^\omega: \Omega^k(M)\oplus\Omega^{k-\ell+1}(M)&\to C^k(f, g)\oplus C^{k-\ell+1}(f, g)\\
        \begin{bmatrix}
            \alpha\\
            \beta
        \end{bmatrix} &\mapsto \begin{bmatrix}
            \Phi & K(\omega)\\
           0  & \Phi 
        \end{bmatrix}\begin{bmatrix}
            \alpha\\
            \beta
        \end{bmatrix}
    \end{split}.
\end{align}
\begin{theorem}[Clausen-Tang-Tseng \cite{clausen_tang_tseng_2024_mappingconemorsetheory}]
    The map $\Phi^\omega$ satisfies 
    \begin{align}\label{mapping cone morse chain map statement}
        \Phi^\omega\circ d^\omega = \partial^\omega\circ \Phi^\omega.
    \end{align}
    Moreover, it induces an isomorphism between the cohomology groups of the two mapping cone complexes.
\end{theorem}

\section{Mapping cone Hodge theory}\label{section of Mapping cone Hodge theory}
We now begin the analytic studies. In this section, we introduce the mapping cone Laplacian and present its deformation. Also, we explain the details about the mapping cone Witten instanton complex (\ref{define the mapping cone witten instanton}).

First, we define norms. 
Let $\nabla$ be the Levi-Civita connection of 
the metric $g$ on $M$. Then for any $\alpha\in\Omega^k(M)$,
\begin{align}
    \nabla^n\alpha = \nabla\nabla\cdots\nabla\alpha
\end{align}
is a smooth section of the bundle 
\begin{align}
   \left(\otimes^n T^*M\right)\otimes\Lambda^kT^*M =  \left(T^*M\otimes T^*M\otimes\cdots\otimes T^*M\right)\otimes\Lambda^kT^*M.
\end{align}
We can naturally view $g$ as a bundle metric on $\left(\otimes^n T^*M\right)\otimes\Lambda^kT^*M$. Let $\dvol_g$ be the volume form on $M$ associated with $g$. By \cite[Part 3, Section 4, Definition 4.6]{handbookofglobalanalysis}, the $n$-th Sobolev norm determined by $g$ is given by 
\begin{align}
    \|\alpha\|_n = \left(\sum_{j = 0}^n\int_M g(\nabla^j\alpha, \nabla^j\alpha)\dvol_g\right)^{1/2}. 
\end{align}
Then, we define the norm $\|\cdot\|_n$ on $\bigoplus_{k = 0}^{m+\ell-1}\left(\Omega^k(M)\oplus\Omega^{k-\ell+1}(M)\right)$ such that for all $0\leqslant k\leqslant m+\ell-1$ and $\begin{bmatrix}
        \alpha\\
        \beta
    \end{bmatrix}\in\Omega^k(M)\oplus\Omega^{k-\ell+1}(M)$, 
\begin{align}\label{sobolev n norm expressions}
    \left\|\begin{bmatrix}
        \alpha\\
        \beta
    \end{bmatrix}\right\|_n\coloneqq \left(\left\|\alpha\right\|_n^2 + \left\|\beta\right\|_n^2\right)^{1/2}.
\end{align}
In addition, when $k\neq k'$, we require that with respect to the inner product induced by $\|\cdot\|_n$, $\Omega^k(M)\oplus\Omega^{k-\ell+1}(M)$ is orthogonal to $\Omega^{k'}(M)\oplus\Omega^{k'-\ell+1}(M)$. 
\begin{lemma}
    The norm $\|\cdot\|_n$ is well-defined on $\bigoplus_{k = 0}^{m+\ell-1}\left(\Omega^k(M)\oplus\Omega^{k-\ell+1}(M)\right)$.
\end{lemma}
\begin{proof}
This norm and its associated inner product are naturally extended to the direct sum of direct sums, so it is well-defined. Here, we double-check the triangle inequality: 

For $\begin{bmatrix}
    \alpha\\
    \beta
\end{bmatrix}\in \Omega^{k}(M)\oplus\Omega^{k-\ell+1}(M)$ and $\begin{bmatrix}
    \alpha'\\
    \beta'
\end{bmatrix}\in \Omega^{k'}(M)\oplus\Omega^{k'-\ell+1}(M)$, when $k\neq k'$, 
    \begin{align}
        \left\|\begin{bmatrix}
        \alpha\\
        \beta
    \end{bmatrix} + \begin{bmatrix}
        \alpha'\\
        \beta'
    \end{bmatrix}\right\|_n = \left(\left\|\begin{bmatrix}
        \alpha\\
        \beta
    \end{bmatrix}\right\|_n^2 + \left\|\begin{bmatrix}
        \alpha'\\
        \beta'
    \end{bmatrix}\right\|_n^2\right)^{1/2}\leqslant\ \left\|\begin{bmatrix}
        \alpha\\
        \beta
    \end{bmatrix}\right\|_n + \left\|\begin{bmatrix}
        \alpha'\\
        \beta'
    \end{bmatrix}\right\|_n.
    \end{align}
    When $k = k'$, we have 
    \begin{align}
        & \left\|\begin{bmatrix}
        \alpha\\
        \beta
    \end{bmatrix} + \begin{bmatrix}
        \alpha'\\
        \beta'
    \end{bmatrix}\right\|_n  \nonumber\\
    =\ &\left(\left\|\alpha+\alpha'\right\|_n^2 + \left\|\beta+\beta'\right\|_n^2\right)^{1/2} \nonumber\\
    \leqslant\ &  \left(\|\alpha\|_n^2 + \|\alpha'\|_n^2 + 2\|\alpha\|_n\|\alpha'\|_n + \|\beta\|_n^2 + \|\beta'\|_n^2 + 2\|\beta\|_n\|\beta'\|_n\right)^{1/2} \nonumber\\
    \leqslant\ & \left(\|\alpha\|_n^2 + \|\beta\|_n^2 + \|\alpha'\|_n^2 + \|\beta'\|_n^2 + 2\sqrt{(\|\alpha\|_n^2+\|\beta\|_n^2)(\|\alpha'\|_n^2+\|\beta'\|_n^2)}\right)^{1/2}  \nonumber\\
    =\ & \left\|\begin{bmatrix}
        \alpha\\
        \beta
    \end{bmatrix}\right\|_n + \left\|\begin{bmatrix}
        \alpha'\\
        \beta'
    \end{bmatrix}\right\|_n.
    \end{align}
    Thus, we obtain a norm. 
\end{proof}

If there is no particular mentioning, we are using the $L^2$ norm $\|\cdot\|_0$. For convenience, we write $\|\cdot\|$ instead of $\|\cdot\|_0$. The $n$-th Sobolev norms will not be used until Section \ref{section of Isomorphism between cochain complexes}. 

For any (locally or globally defined) vector field $X$ on $M$, we let 
$X^*\coloneqq g(X,\cdot)$ be the dual $1$-form. The Clifford actions of $X$ on any form $\alpha$ are given by 
\begin{align}
    c(X)\alpha \coloneqq X^*\wedge\alpha - X\lrcorner\alpha\ \ \text{and}\ \  \hat{c}(X)\alpha \coloneqq X^*\wedge\alpha + X\lrcorner\alpha\ . 
\end{align}
For a $1$-form $\eta$, we have its dual vector field $Y$ given by 
$g(Y,\cdot) = \eta$. Then, we let 
\begin{align}
    c(\eta)\alpha \coloneqq \eta\wedge\alpha - Y\lrcorner\alpha\ \ \text{and}\ \  \hat{c}(\eta)\alpha \coloneqq \eta\wedge\alpha + Y\lrcorner\alpha\ . 
\end{align}
With the Morse function $f$ and the parameters $S>0$ and $T\geqslant 0$, the Witten deformation of $d^\omega$ is given by 
\begin{align}\label{deformed version of the mapping cone de Rham cochain complex}
   \begin{split}
       d^\omega_{ST}: \Omega^k(M)\oplus\Omega^{k-\ell+1}(M)&\to\Omega^{k+1}(M)\oplus\Omega^{k-\ell+2}(M)\\
    \begin{bmatrix}
        \alpha\\
        \beta
    \end{bmatrix} &\mapsto \begin{bmatrix}
        d+Tdf & S^{-1}\omega\\
        0 & (-1)^{\ell-1}(d+Tdf)
    \end{bmatrix}\begin{bmatrix}
        \alpha\\
        \beta
    \end{bmatrix}.
   \end{split}
\end{align}
Here, we omit the ``$\wedge$'' after $df$ and $\omega$. Let 
\begin{align}
\begin{split}
    \varrho_{ST}: \Omega^k(M)\oplus\Omega^{k-\ell+1}(M)&\to\Omega^{k+1}(M)\oplus\Omega^{k-\ell+2}(M)\\
    \begin{bmatrix}
        \alpha\\
        \beta
    \end{bmatrix} &\mapsto \begin{bmatrix}
        e^{Tf}\alpha\\
        S^{-1}e^{Tf}\beta
    \end{bmatrix}.
\end{split}
\end{align}
It is straightforward to check: 
\begin{proposition}
\label{routinely check the deformation does not change the cohomology}
For any $S>0$ and $T\geqslant 0$, the diagram
\begin{equation}
    \begin{tikzcd}
\Omega^k(M)\oplus\Omega^{k-\ell+1}(M) \arrow{r}{d^\omega_{ST}} \arrow[swap]{d}{\varrho_{ST}} & \Omega^{k+1}(M)\oplus\Omega^{k-\ell+2}(M) \arrow{d}{\varrho_{ST}} \\%
\Omega^k(M)\oplus\Omega^{k-\ell+1}(M) \arrow{r}{d^\omega}& \Omega^{k+1}(M)\oplus\Omega^{k-\ell+2}(M)
\end{tikzcd}
\end{equation}
commutes. In addition, since $\varrho_{ST}$ is invertible, $\varrho_{ST}$ is a cochain isomorphism.  
\begin{proof}
    We notice that 
    \begin{align}
    \begin{split}
         d^\omega\varrho_{ST}\begin{bmatrix}
            \alpha\\
            \beta
        \end{bmatrix} = \begin{bmatrix}
            e^{Tf} Tdf\wedge\alpha + e^{Tf}d\alpha + e^{Tf}\omega\wedge S^{-1}\beta\\
            (-1)^{\ell-1} S^{-1}(e^{Tf}Tdf\wedge\beta + e^{Tf}d\beta)
        \end{bmatrix} = \varrho_{ST} d^\omega_{ST}\begin{bmatrix}
            \alpha\\
            \beta
        \end{bmatrix}.
    \end{split}
    \end{align}
    Thus, the diagram commutes, and $\varrho_{ST}$ is a cochain isomorphism. 
\end{proof}
\end{proposition}
Therefore, the cochain complex (\ref{deformed version of the mapping cone de Rham cochain complex}) computes the same cohomology as (\ref{mapping cone de Rham complex}) does. We will  use (\ref{deformed version of the mapping cone de Rham cochain complex}) in the rest of this paper. 

\begin{proposition}
    For all $S>0$ and $T\geqslant 0$, the cochain complex {\normalfont(\ref{deformed version of the mapping cone de Rham cochain complex})} is an elliptic complex.
\end{proposition}
\begin{proof}
    Let $h$ be a smooth function on $M$. Then, for all 
    $\begin{bmatrix}
        \alpha\\
        \beta
    \end{bmatrix}\in\Omega^k(M)\oplus\Omega^{k-\ell+1}(M)$,
    \begin{align}
        d^\omega_{ST} \begin{bmatrix}
            h\alpha\\
            h\beta
        \end{bmatrix} - h d^\omega_{ST}\begin{bmatrix}
            \alpha\\
            \beta
        \end{bmatrix} = \begin{bmatrix}
            dh\wedge\alpha \\
            (-1)^{\ell-1}dh\wedge\beta
        \end{bmatrix}. 
    \end{align}
    Thus, for each $x\in M$ and $\xi\in T^*_x M\backslash\{0\}$, we have the principal symbol 
    \begin{align}\label{proving the elliptic complex}
    \begin{split}
        \sigma(d_{ST}^\omega)(x,\xi): \Lambda^k T^*_x M \oplus \Lambda^{k-\ell+1} T^*_x M &\to \Lambda^{k+1} T^*_x M \oplus \Lambda^{k-\ell+2} T^*_x M\\
        \begin{bmatrix}
            \theta_1\\
            \theta_2
        \end{bmatrix}&\mapsto\begin{bmatrix}
            \xi\wedge \theta_1\\
            (-1)^{\ell-1}\xi\wedge \theta_2
        \end{bmatrix}
        \end{split}
    \end{align}
    of $d_{ST}^\omega$. By the exact sequence \cite[Example 10.4.29]{nicolaescu2020lectures}
    \begin{align}
   0\xlongrightarrow[]{}\mathbb{R}\xlongrightarrow[]{\xi\wedge}\Lambda^0 T_x^*M \xlongrightarrow[]{\xi\wedge}\Lambda^1 T_x^*M \xlongrightarrow[]{\xi\wedge}\cdots\xlongrightarrow[]{\xi\wedge}\Lambda^m T_x^*M\xlongrightarrow[]{}0, 
    \end{align}
    the map (\ref{proving the elliptic complex}) also defines an exact sequence. Thus, (\ref{deformed version of the mapping cone de Rham cochain complex}) is elliptic.
\end{proof}

Therefore, based on the Hodge theory \cite[Section 10.4.3]{nicolaescu2020lectures} of elliptic complexes, 
we are able to define the following instanton complex. 
Let 
\begin{align}
    \mathbb{D}_{ST} = d^\omega_{ST} + {d^\omega_{ST}}^*.
\end{align}
Immediately, we see that $\mathbb{D}_{ST}$ is of the form
\begin{align}
    \mathbb{D}_{ST} = \begin{bmatrix}
        d+d^*+T\hat{c}(df) & S^{-1}\omega\\
        S^{-1}\omega^* & (-1)^{\ell-1}(d+d^*+T\hat{c}(df))
    \end{bmatrix},
\end{align}
where $\omega^*$ is the adjoint of $\omega\wedge$ with respect to the $L^2$ norm on differential forms on $M$. We sometimes write it as $\omega^*\lrcorner$ since it acts like the interior multiplication. 
Then, we let
\begin{align}
	F_{ST}^k(f,g,\omega)\coloneqq\bigoplus_{0\leqslant\lambda\leqslant 1}\left\{\mathbf{w}\in\Omega^k(M)\oplus\Omega^{k-\ell+1}(M): \mathbb{D}_{ST}^2\mathbf{w} = \lambda\mathbf{w}\right\}. 
\end{align}
In other words, $F_{ST}^k(f,g,\omega)$ is the direct sum of eigenspaces associated with the eigenvalues of the Laplacian $\mathbb{D}_{ST}^2$ in $[0,1]$. 
Since $d_{ST}^\omega$ commutes with $\mathbb{D}_{ST}^2$, the restriction of $d_{ST}^\omega$ to $F_{ST}^k(f,g,\omega)$ defines a cochain complex
\begin{align}\label{revisit the instanton}
    d_{ST}^\omega: F_{ST}^k(f,g,\omega)\to F_{ST}^{k+1}(f,g,\omega)\ \ (-1\leqslant k\leqslant m+\ell-1).
    \end{align}
\begin{definition}\normalfont
    We call (\ref{revisit the instanton}) the instanton cochain complex of $(f,g,\omega)$ with parameters $S$ and $T$. 
\end{definition}
We prove that the left wedge by $\omega$ is bounded. Let 
\begin{align}
    \Omega^\bullet(M)\coloneqq \bigoplus_{k = -1}^m\Omega^k(M).
\end{align}
We have the following refined version of \cite[Theorem 6.18(i)]{warner2013foundations} at the $L^2$ norm.
\begin{lemma}\label{lemma like warner's theorem 6.18 i}
Let $
    a(\omega) = 2^m\cdot\max_{x\in M}\sqrt{g(\omega,\omega)(x)}
$. Then, we have 
\begin{align}
	\|\omega\wedge\eta\|\leqslant a(\omega)\|\eta\|
\end{align}
    for all $\eta\in\Omega^\bullet(M)$. 
\end{lemma}
\begin{proof}
Let $\{U_\alpha\}_{\alpha}$ be a finite open cover of $M$, and let ${\rho_\alpha}$ be the associated partition of unity with $\supp(\rho_\alpha)\subseteq U_\alpha$. Then, for each $\alpha$, we choose an orthonormal local coframe $\theta_1^\alpha, \cdots, \theta_{ m}^\alpha$ and write
\begin{align}
    \omega|_{U_\alpha} = \sum_{1\leqslant j_1<\cdots<j_\ell\leqslant m}a_{j_1\cdots j_\ell}^\alpha \theta_{j_1}^\alpha\wedge\cdots\wedge \theta_{j_\ell}^\alpha
\end{align}
and 
\begin{align} 
    \eta|_{U_\alpha} = \sum_{r = 0}^{ m}\sum_{1\leqslant i_1<\cdots<i_r\leqslant m}b_{i_1\cdots i_r}^\alpha \theta_{i_1}^\alpha\wedge\cdots\wedge \theta_{i_r}^\alpha. 
\end{align} 
Let $\dvol$ be the volume form on $M$ given by $\dvol|_{U_\alpha} = \theta_1^\alpha\wedge\cdots\wedge\theta_{m}^\alpha$. Then, 
\begin{align}
   \|\omega\wedge\eta\|^2
    =\ & \sum_\alpha\int_{U_\alpha} \rho_\alpha g(\omega\wedge\eta, \omega\wedge\eta) \dvol \nonumber\\
    \leqslant\ & \sum_\alpha\int_{U_\alpha} \rho_\alpha \left(\sum_{r = 0}^{ m}\ \sum_{1\leqslant i_1<\cdots<i_r\leqslant m}\ \sum_{1\leqslant j_1<\cdots<j_\ell\leqslant m} a_{j_1\cdots j_\ell}^\alpha b_{i_1\cdots i_r}^\alpha \right)^2 \dvol \nonumber\\
    \leqslant\ & \sum_\alpha\int_{U_\alpha} \rho_\alpha  \sum_{1\leqslant j_1<\cdots<j_\ell\leqslant m} 2^{ m}\left(a_{j_1\cdots j_\ell}^\alpha\right)^2  \sum_{r = 0}^{ m}\ \sum_{1\leqslant i_1<\cdots<i_r\leqslant m}\binom{ m}{\ell}\left(b_{i_1\cdots i_r}^\alpha \right)^2 \dvol \nonumber\\
    \leqslant\ & (a(\omega))^2\sum_\alpha\int_{U_\alpha} \rho_\alpha    \sum_{r = 0}^{ m}\ \sum_{1\leqslant i_1<\cdots<i_r\leqslant m}\left(b_{i_1\cdots i_r}^\alpha \right)^2 \dvol \nonumber\\
   % =\ & (a(\omega))^2\sum_\alpha\int_{U_\alpha} \rho_\alpha |\eta|_{U_\alpha}^2 \dvol \nonumber\\
    =\ & (a(\omega))^2 g(\eta,\eta).
\end{align}
Thus, $\|\omega\wedge\eta\|\leqslant a(\omega)\|\eta\|$. 
\end{proof}

\begin{remark}\normalfont
    Since $M$ is closed, under the Sobolev norm $\|\cdot\|_n$, similar boundedness appears for the interior multiplications and the Clifford actions of any other forms. This is again guaranteed by \cite[Theorem 6.18(i)]{warner2013foundations}, At $n = 0$, we need the refined Lemma \ref{lemma like warner's theorem 6.18 i}. At $n\geqslant 1$, \cite[Theorem 6.18(i)]{warner2013foundations} is enough in this paper. 
\end{remark}

\section{Harmonic oscillators around critical points}\label{section of Harmonic oscillators around critical points}
In this section, we review several conclusions \cite[Chapter 5]{wittendeformationweipingzhang} about harmonic oscillators in the classical Morse theory. Then, we use these conclusions to study the behavior of eigenvalues of $\mathbb{D}_{ST}^2$ for sufficiently large $S$ and $T$. 

Let $p\in Z^k(df)$. Around $p$, we adjust $(f, g)$ to obtain a coordinate chart 
\begin{align}
	U \coloneqq U(\varepsilon) = \{(x_1,\cdots,x_m): x_1^2+\cdots+x_m^2<\varepsilon\} 
\end{align} 
centered at $p$ such that these $U$'s are disjoint, and 
\begin{align}\label{morse lemma coordinate expression of f}
       f(x_1,\cdots,x_m) =\ & f(p)-\dfrac{1}{2}(x_1^2+\cdots+x_k^2) + \dfrac{1}{2}(x_{k+1}^2+\cdots+x_m^2),\\
   	g|_U =\ & dx_1^2+\cdots + dx_m^2.
   \end{align}
   Like in \cite[(7.6)]{Bismut_Zhang_1992_theorem_Cheeger_Muller}, \cite[(5.3)]{Bismut1994}, By \cite[Proposition 5.1]{Helffer01011985}, we ensure that $(f,g)$ is still a Morse-Smale pair. 
   
   Let $\Delta$ be the Laplace-Beltrami operator on smooth functions, and 
   \begin{align}
       D_T = d+d^*+T\hat{c}(df).
   \end{align}
   By \cite[(5.14)]{wittendeformationweipingzhang},  
   \begin{align}\label{local expression of harmonic oscillator correct way to read}
   	 D_T^2|_U = \Delta - mT + T^2(x_1^2+\cdots+x_m^2) + 2T\sum_{i = 1}^k \partial_i\lrcorner dx_i\wedge + 2T\sum_{i = 1}^k dx_i\wedge\partial_i\lrcorner.
   \end{align}
   The correct way to read (\ref{local expression of harmonic oscillator correct way to read}) is, given a form 
   \begin{align}
   	\alpha = \varphi dx_{j_1}\wedge\cdots\wedge dx_{j_r} \in\Omega^r(U), 
   \end{align}
   we have 
   \begin{align}
   \begin{split}
   	D_T^2\alpha =\ & (\Delta\varphi)dx_{j_1}\wedge\cdots\wedge dx_{j_r} - mT\varphi dx_{j_1}\wedge\cdots\wedge dx_{j_r} + T^2(x_1^2+\cdots+x_m^2)\varphi dx_{j_1}\wedge\cdots\wedge dx_{j_r}\\
   	&  + \varphi\cdot 2T\sum_{i = 1}^k \partial_i\lrcorner (dx_i\wedge dx_{j_1}\wedge\cdots\wedge dx_{j_r}) + \varphi\cdot 2T\sum_{i = 1}^k dx_i\wedge\partial_i\lrcorner(dx_{j_1}\wedge\cdots\wedge dx_{j_r}).
   	\end{split}
   \end{align}
   Let $\gamma$ be the smooth bump function such that $\supp(\gamma)\subseteq U(\varepsilon/2)$ and 
   \begin{align}
   	\gamma|_{U(\varepsilon/4)} = 1. 
   \end{align}
   Let 
   \begin{align}\label{The normal distribution}
   	\rho = \exp\left(-\dfrac{T}{2}(x_1^2+\cdots+x_m^2)\right)dx_1\wedge\cdots\wedge dx_k.
   \end{align}
   The form $\gamma\rho$ is globally defined on $M$. 
   Let $\xi_p$ be the form $\gamma\rho$ associated with the critical point $p$. Then, we define the space 
   \begin{align}
       E_T = \text{span}_{\mathbb{R}}\{\xi_p: p\in Z(df)\}.
   \end{align}
   Then, since $E_T$ is finite dimensional, we let $E_T^\perp$ be the orthogonal complement of $E_T$ in $\Omega^\bullet(M)$ with respect to the $L^2$ norm. 

   Recall \cite[Proposition 5.6(iii)]{wittendeformationweipingzhang}: 
   \begin{proposition}
   	There exist constants $C_1>0$ and $T_1>0$ such that when $T>T_1$, 
    \begin{align}
        \|D_T\alpha\|\geqslant C_1\sqrt{T}\|\alpha\|
    \end{align}
    for all $\alpha\in E_T^\perp$. 
   \end{proposition}

   The next proposition refines \cite[Proposition 5.6(i)(ii)]{wittendeformationweipingzhang}: 

   \begin{proposition}
       There exist constants $C_2>0$ and $T_2>0$ such that when  
       $T>T_2$, 
       \begin{align}\label{refine the DT,2}
           \|D_T\alpha\|\leqslant C_2 e^{-C_2 T}\|\alpha\|
       \end{align}
       for all $\alpha\in E_T$. 
   \end{proposition}
   \begin{proof}
       We normalize $\xi_p$ to $\|\xi_p\|^{-1}\xi_p$. Then, 
       \begin{align}
             D_T(\|\xi_p\|^{-1}\xi_p) = \|\xi_p\|^{-1} c(d\gamma)\rho.
       \end{align}
       By the Gaussian integral, we find that when $T$ is sufficiently large, 
       \begin{align}
      \dfrac{9}{10}\left(\dfrac{\pi}{T}\right)^{m/4} \leqslant \|\xi_p\|\leqslant \left(\dfrac{\pi}{T}\right)^{m/4}. 
   \end{align}
       Since $\supp(\gamma)\subseteq U(\varepsilon/2)$ and $\gamma|_{U(\varepsilon/4)} = 1$, we see that 
       \begin{align}
           \|c(d\gamma)\rho\|\leqslant C_2'e^{-C_2'T}.
       \end{align}
       Then, we obtain (\ref{refine the DT,2}) by absorbing $T^{m/4}$ into $e^{-C_2'T}$ when $T>T_2$.
   \end{proof}

   Now, like \cite[(9.12)]{bismutandlebeau} and \cite[(5.18)]{wittendeformationweipingzhang}, we define the space 
   \begin{align}\label{the definition of E_T space here}
       \mathbb{E}_T = \text{span}_{\mathbb{R}}\left\{\begin{bmatrix}
           \xi_p\\
           0
       \end{bmatrix}, \begin{bmatrix}
           0\\
           \xi_p
       \end{bmatrix}: p\in Z(df)\right\}\subseteq\bigoplus_{i = -1}^{m+\ell-1}\left(\Omega^{i}(M)\oplus\Omega^{i-\ell+1}(M)\right).
   \end{align}
   Let $\mathbb{E}_T^\perp$ be the orthogonal complement of $\mathbb{E}_T$ in 
   \begin{align}
       \bigoplus_{i = -1}^{m+\ell-1}\Omega^{i}(M)\oplus\Omega^{i-\ell+1}(M)
   \end{align}
   with respect to the $L^2$ norm. 
   Then, we find 
   \begin{proposition}\label{lower estimate of eigenvalues}
       When $T>T_2$, we have
       \begin{align}
           \left\|\mathbb{D}_{ST}\begin{bmatrix}
               \alpha\\
               \beta
           \end{bmatrix}\right\|\leqslant \sqrt{2}(C_2 e^{-C_2 T} + S^{-1}a(\omega))\left\|\begin{bmatrix}
               \alpha\\
               \beta
           \end{bmatrix}\right\|
       \end{align}
       for all $\begin{bmatrix}
               \alpha\\
               \beta
           \end{bmatrix}\in\mathbb{E}_T$. 
   \end{proposition}
   \begin{proof}
       Since $\alpha$ and $\beta$ are elements in $E_T$, we apply previous conclusions and get 
       \begin{align}
           & \left\|\mathbb{D}_{ST}\begin{bmatrix}
               \alpha\\
               \beta
           \end{bmatrix}\right\| \nonumber\\
           =\ & \left(\|D_T\alpha + S^{-1}\omega\wedge\beta\|^2 + \left\|S^{-1}\omega^*\lrcorner\alpha + (-1)^{\ell-1} D_T\beta\right\|^2\right)^{1/2} \nonumber\\
           \leqslant\ & \|D_T\alpha\| + \|S^{-1}\omega\wedge\beta\| + \|S^{-1}\omega^*\lrcorner\alpha \|+ \|D_T\beta\| \nonumber\\
           \leqslant\ &  \left(C_2 e^{-C_2 T} + S^{-1} a(\omega)\right)\left(\|\alpha\|+\|\beta\|\right) \nonumber\\
           \leqslant\ &  \sqrt{2}\left(C_2 e^{-C_2 T} + S^{-1} a(\omega)\right)\left\|\begin{bmatrix}
               \alpha\\
               \beta
           \end{bmatrix}\right\|
       \end{align}
       for all $\begin{bmatrix}
               \alpha\\
               \beta
           \end{bmatrix}\in\mathbb{E}_T$. 
   \end{proof}

Notice that $\begin{bmatrix}
               \alpha\\
               \beta
           \end{bmatrix}\in\mathbb{E}_T^\perp$ if and only if 
           \begin{align}
               \alpha\in E_T^\perp\ \ \text{and}\ \ \beta\in E_T^\perp
           \end{align}
           at the same time. Then, we find
   \begin{proposition}\label{upper estimate of eigenvalues}
       When $T>T_1$, we have 
       \begin{align}\label{lowerbound}
           \left\|\mathbb{D}_{ST}\begin{bmatrix}
               \alpha\\
               \beta
           \end{bmatrix}\right\|\geqslant \left(\dfrac{1}{\sqrt{2}}C_1\sqrt{T} -\dfrac{1}{\sqrt{2}}S^{-1} a(\omega)\right)\left\|\begin{bmatrix}
               \alpha\\
               \beta
           \end{bmatrix}\right\|
       \end{align}
       for all $\begin{bmatrix}
               \alpha\\
               \beta
           \end{bmatrix}\in\mathbb{E}_T^\perp$. 
   \end{proposition}
   \begin{proof}
   We find that 
       \begin{align}
           & \left\|\mathbb{D}_{ST}\begin{bmatrix}
               \alpha\\
               \beta
           \end{bmatrix}\right\|   \nonumber\\
           =\ & \left(\|D_T\alpha + S^{-1}\omega\wedge\beta\|^2 + \left\|S^{-1}\omega^*\lrcorner\alpha + (-1)^{\ell-1} D_T\beta\right\|^2\right)^{1/2}   \nonumber\\
           \geqslant\ & \dfrac{1}{\sqrt{2}}\|D_T\alpha\| -  \dfrac{1}{\sqrt{2}}\|S^{-1}\omega\wedge\beta\| + \dfrac{1}{\sqrt{2}}\|D_T\beta\|- \dfrac{1}{\sqrt{2}}\|S^{-1}\omega^*\lrcorner\alpha \|   \nonumber\\
           \geqslant\ &  \left(\dfrac{1}{\sqrt{2}}C_1\sqrt{T} -\dfrac{1}{\sqrt{2}}S^{-1} a(\omega)\right)\left(\|\alpha\|+\|\beta\|\right)   \nonumber\\
           \geqslant\ &  \left(\dfrac{1}{\sqrt{2}}C_1\sqrt{T} -\dfrac{1}{\sqrt{2}}S^{-1} a(\omega)\right)\left\|\begin{bmatrix}
               \alpha\\
               \beta
           \end{bmatrix}\right\|.
       \end{align}
       Thus, we obtain (\ref{lowerbound}).
\end{proof}
       Like \cite[Lemma 5.3]{weipingzhangnewedition}, we summarize these estimates of norms into: 
       \begin{proposition}\label{eigenvalue behavior makes some operators invertible}
           The eigenvalues of $\mathbb{D}_{ST}^2$
           belong to the disjoint union 
           \begin{align}
               \left[0, \left(\sqrt{2}(C_2 e^{-C_2 T} + S^{-1}a(\omega))\right)^2\right]\cup \left[\left(\dfrac{1}{\sqrt{2}}C_1\sqrt{T} -\dfrac{1}{\sqrt{2}}S^{-1} a(\omega)\right)^2, +\infty\right)
           \end{align}
           of two intervals when $S$ and $T$ are sufficiently large. 
       \end{proposition}
   \begin{proof}
    Suppose we have $\mathbf{w}$ and $\lambda$ such that 
    \begin{align}
        \left(\sqrt{2}(C_2 e^{-C_2 T} + S^{-1}a(\omega))\right)^2<\lambda<\ & \left(\dfrac{1}{\sqrt{2}}C_1\sqrt{T} -\dfrac{1}{\sqrt{2}}S^{-1} a(\omega)\right)^2,\\
        D_{ST}^2\mathbf{w} =\ & \lambda\mathbf{w}.
    \end{align}
    Then, we write $\mathbf{w} = \mathbf{w}_1+\mathbf{w}_2$ with $\mathbf{w}_1\in \mathbb{E}_T$ and $\mathbf{w}_2\in \mathbb{E}_T^\perp$. Let $\langle\cdot,\cdot\rangle$ be the inner product induced by the $L^2$ norm $\|\cdot\|$. Since $\mathbb{D}_{ST}$ is self-adjoint, we see that 
    \begin{align}
       0  =\ & \left\langle (\mathbb{D}_{ST}^2-\lambda)\mathbf{w}, \mathbf{w}_1-\mathbf{w}_2 \right\rangle    \nonumber\\
        =\ & \left\langle (\mathbb{D}_{ST}^2-\lambda)\mathbf{w}_1,\mathbf{w}_1 \right\rangle - \left\langle (\mathbb{D}_{ST}^2-\lambda)\mathbf{w}_2,\mathbf{w}_2 \right\rangle   \nonumber\\
        \leqslant\ & \left(\left(\sqrt{2}(C_2 e^{-C_2 T} + S^{-1}a(\omega))\right)^2-\lambda\right)\|\mathbf{w}_1\|^2    \nonumber\\
        & + \left(\lambda - \left(\dfrac{1}{\sqrt{2}}C_1\sqrt{T} -\dfrac{1}{\sqrt{2}}S^{-1} a(\omega)\right)^2\right) \|\mathbf{w}_2\|^2   \nonumber\\
        \leqslant\ &  0.
    \end{align}
    Therefore, $\mathbf{w}$ must be zero.  
\end{proof}
   
\section{Isomorphism between cochain complexes}\label{section of Isomorphism between cochain complexes}
In this section, we construct and prove the cochain isomorphism required in Theorem 
\ref{main theorem}. We follow the process like \cite[Chapter 6]{wittendeformationweipingzhang} to use the orthogonal projection to prove the main result. Essentially speaking, this section is a refinement of the previous section to a large extent. 

Let $I$ be the identity operator, and $\mathbb{S}^1$ be the unit circle centered at the origin. Because of Proposition \ref{eigenvalue behavior makes some operators invertible},  
we can define
\begin{align}
    P_{ST} \begin{bmatrix}
        \alpha\\
        \beta
    \end{bmatrix} = \dfrac{1}{2\pi\sqrt{-1}}\int_{z\in\mathbb{S}^1} (zI - \mathbb{D}_{ST})^{-1}\begin{bmatrix}
        \alpha\\
        \beta
    \end{bmatrix}\ dz
\end{align}
for $\alpha\in\Omega^k(M)$ and $\beta\in\Omega^{k-\ell+1}(M)$. 
This $P_T$ is the orthogonal projection 
\begin{align}
    P_{ST}: \bigoplus_{k = -1}^{m+\ell-1}\left(\Omega^k(M)\oplus\Omega^{k-\ell+1}(M)\right) \to \bigoplus_{k = -1}^{m+\ell-1} F_{ST}^k(f,g,\omega)
\end{align}
onto $\bigoplus_{k = -1}^{m+\ell-1} F_{ST}^k(f,g,\omega)$. Also, for any $p\in Z^k(df)$ and $q\in Z^{k-\ell+1}(df)$, we let 
\begin{align}
    J_T\begin{bmatrix}
        p\\
        0
    \end{bmatrix} = \|\xi_p\|^{-1}\begin{bmatrix}
        \xi_p\\
        0
    \end{bmatrix},\ \ J_T\begin{bmatrix}
        0\\
        q
    \end{bmatrix} = \|\xi_q\|^{-1}\begin{bmatrix}
        0\\
        \xi_q
    \end{bmatrix},
\end{align}
and obtain a map
\begin{align}
\begin{split}
    J_T: \bigoplus_{k = -1}^{m+\ell-1}\left(C^k(f,g)\oplus C^{k-\ell+1}(f,g)\right) \to \mathbb{E}_T.
\end{split}
\end{align}

\begin{proposition}\label{a preliminary prerequisite for the final isomorphism but this is just the very rough one}
The map $P_{ST}\circ J_T$ is an isomorphism when $S$ and $T$ are sufficiently large. 
\end{proposition}
\begin{proof}
    For any $p\in Z^k(df)$, 
    \begin{align}\label{rough estimate of projection}
        & \left\|P_{ST}\circ J_T\begin{bmatrix}
            p\\
            0
        \end{bmatrix} - J_T\begin{bmatrix}
            p\\
            0
        \end{bmatrix}\right\|   \nonumber\\
        \leqslant \ & \left\|\left(\dfrac{1}{\sqrt{2}}C_1\sqrt{T} -\dfrac{1}{\sqrt{2}}S^{-1} a(\omega)\right)^{-1} \mathbb{D}_{ST}\left(P_{ST}\left(\|\xi_p\|^{-1}\begin{bmatrix}
            \xi_p\\
            0
        \end{bmatrix}\right) - \|\xi_p\|^{-1}\begin{bmatrix}
            \xi_p\\
            0
        \end{bmatrix}\right)\right\|   \nonumber\\
        \leqslant\ & \left(\dfrac{1}{\sqrt{2}}C_1\sqrt{T} -\dfrac{1}{\sqrt{2}}S^{-1} a(\omega)\right)^{-1}\left\| \mathbb{D}_{ST}P_{ST}\left(\|\xi_p\|^{-1}\begin{bmatrix}
            \xi_p\\
            0
        \end{bmatrix}\right)\right\|   \nonumber\\
        & +\left(\dfrac{1}{\sqrt{2}}C_1\sqrt{T} -\dfrac{1}{\sqrt{2}}S^{-1} a(\omega)\right)^{-1}\left\|\mathbb{D}_{ST}\left(\|\xi_p\|^{-1}\begin{bmatrix}
            \xi_p\\
            0
        \end{bmatrix}\right)\right\|   \nonumber\\
        \leqslant\ & \left(\dfrac{1}{\sqrt{2}}C_1\sqrt{T} -\dfrac{1}{\sqrt{2}}S^{-1} a(\omega)\right)^{-1}\sqrt{2}(C_2 e^{-C_2 T} + S^{-1}a(\omega)) \left(\|\xi_p\|^{-1}\left\|P_{ST}\begin{bmatrix}
            \xi_p\\
            0
        \end{bmatrix}\right\| + 1\right)   \nonumber\\
        \leqslant\ & \left(\dfrac{1}{\sqrt{2}}C_1\sqrt{T} -\dfrac{1}{\sqrt{2}}S^{-1} a(\omega)\right)^{-1}\sqrt{2}(C_2 e^{-C_2 T} + S^{-1}a(\omega)) \left(1 + 1\right).
    \end{align}
   The last line of (\ref{rough estimate of projection}) goes to $0$ when $S\to+\infty$ and $T\to+\infty$. The same calculation applies to $\begin{bmatrix}
        0\\
        q
    \end{bmatrix}$ for $q\in\Omega^{k-\ell+1}(M)$. 
    Since 
    \begin{align}
        \left\{J_T\begin{bmatrix}
        p\\
        0
    \end{bmatrix},\  J_T\begin{bmatrix}
        0\\
        p
    \end{bmatrix}: p\in Z(df)\right\}
    \end{align}
    is an orthonormal basis of $\mathbb{E}_T$, then $P_{ST}\circ J_T$ is injective when $S$ and $T$ are sufficiently large. 

    Now, we show that $P_{ST}\circ J_T$ is surjective. 
    Suppose that $\mathbf{w}\in\bigoplus_{k = -1}^{m+\ell-1} F_{ST}^k(f,g,\omega)$ is $L^2$ orthogonal to $P_{ST}\mathbb{E}_T$. Then, we write $\mathbf{w} = \mathbf{w}_1+\mathbf{w}_2$ with $\mathbf{w}_1\in \mathbb{E}_T$ and $\mathbf{w}_2\in \mathbb{E}_T^\perp$. Since $F_{ST}^k(f,g,\omega)$ only involves eigenvalues in $[0,1]$, we have 
    \begin{align}\label{estimate added 1}
        \|\mathbf{w}\| \geqslant \ & \|\mathbb{D}_{ST}\mathbf{w}\|   \nonumber\\
        \geqslant\ &  \|\mathbb{D}_{ST}\mathbf{w}_2\| - \|\mathbb{D}_{ST}\mathbf{w}_1\|   \nonumber\\
        \geqslant\ &  \left(\dfrac{1}{\sqrt{2}}C_1\sqrt{T} -\dfrac{1}{\sqrt{2}}S^{-1} a(\omega)\right)\|\mathbf{w}_2\| - \sqrt{2}(C_2 e^{-C_2 T} + S^{-1}a(\omega))\|\mathbf{w}_1\|  \nonumber\\
        \geqslant\ & \left(\dfrac{1}{\sqrt{2}}C_1\sqrt{T} -\dfrac{1}{\sqrt{2}}S^{-1} a(\omega)\right)\|\mathbf{w}_2\| - \sqrt{2}(C_2 e^{-C_2 T} + S^{-1}a(\omega))\|\mathbf{w}\|.
    \end{align}
    Since $\mathbf{w}$ is orthogonal to $P_{ST}\mathbf{w}_1$, we have 
    \begin{align}
        \|\mathbf{w} - P_{ST}\mathbf{w}_1\|\geqslant \|\mathbf{w}\|.
    \end{align}
    Therefore, by (\ref{rough estimate of projection}), we have a constant $\widetilde{C}_1$ such that 
    \begin{align}\label{estimate added 2}
        \|\mathbf{w}_2\| \geqslant  \|\mathbf{w} - P_{ST}\mathbf{w}_1\| - \|P_{ST}\mathbf{w}_1 - \mathbf{w}_1\| 
        \geqslant  \|\mathbf{w}\| - \widetilde{C}_1 T^{-1/2}\|\mathbf{w}_1\|.
    \end{align}
    Then, by (\ref{estimate added 1}) and (\ref{estimate added 2}), we have a constant $\widetilde{C}_2$ such that 
    \begin{align}\label{added estimate 3}
        \|\mathbf{w}\| 
        \geqslant\ & \left(\dfrac{1}{\sqrt{2}}C_1\sqrt{T} -\dfrac{1}{\sqrt{2}}S^{-1} a(\omega)\right)\left(\|\mathbf{w}\|-\widetilde{C}_1 T^{-1/2}\|\mathbf{w}_1\|\right) - \sqrt{2}(C_2 e^{-C_2 T} + S^{-1}a(\omega))\|\mathbf{w}\|   \nonumber\\
        \geqslant\ & \widetilde{C}_2\sqrt{T}\left(\|\mathbf{w}\| - \widetilde{C}_1 T^{-1/2}\|\mathbf{w}_1\|\right) - \|\mathbf{w}\|   \nonumber\\
        \geqslant\ & \widetilde{C}_2\sqrt{T}\|\mathbf{w}\| - \widetilde{C}_2\widetilde{C}_1 \|\mathbf{w}\| - \|\mathbf{w}\|.
    \end{align}
    We see that (\ref{added estimate 3}) holds only when $\mathbf{w}$ is zero. Thus, $P_{ST}\circ J_T$ is surjective.  
\end{proof}

 The isomorphism $P_{ST}\circ J_T$ is not a cochain isomorphism. The projection $P_{ST}$ brings subtle perturbations, so $P_{ST}\circ J_T$ does not guarantee a commutative diagram between cochain complexes. We will use $P_{ST}\circ J_T$ as an auxiliary tool to verify the true cochain isomorphism.

We recall the cochain map $\Phi^\omega$ given by 
(\ref{mapping cone morse and mapping cone de rham quasi-isom expression}). We perturb it using $S$ and $T$ and restrict it to $\bigoplus_{k = -1}^{m+\ell-1} F_{ST}^k(f,g,\omega)$. Thus, we define
\begin{align}
    \begin{split}
        \Phi^\omega_{ST}: F_{ST}^k(f,g,\omega)&\to C^k(f, g)\oplus C^{k-\ell+1}(f, g)\\
        \mathbf{w} &\mapsto \begin{bmatrix}
            \Phi & S^{-1}K(\omega)\\
           0  & S^{-1}\Phi
        \end{bmatrix}\begin{bmatrix}
            e^{Tf} & \\
            & e^{Tf}
        \end{bmatrix}\mathbf{w}\ .
    \end{split}
\end{align} 
\begin{proposition}
    This $\Phi^\omega_{ST}$ is a chain map for all $S>0$ and $T>0$. 
\end{proposition}
\begin{proof}
    We need to verify that 
    \begin{align}\label{perturbed chain map verification process}
        \begin{bmatrix}
            \partial & \cupc(\omega)\\
            0 & (-1)^{\ell-1}\partial
        \end{bmatrix}\begin{bmatrix}
            \Phi e^{Tf} & S^{-1}K(\omega)e^{Tf}\\
           0  & S^{-1}\Phi e^{Tf}
        \end{bmatrix} = \begin{bmatrix}
            \Phi e^{Tf} & S^{-1}K(\omega)e^{Tf}\\
           0  & S^{-1}\Phi e^{Tf}
        \end{bmatrix}\begin{bmatrix}
            d+Tdf & S^{-1}\omega\\
            0 & (-1)^{\ell-1}(d+Tdf)
        \end{bmatrix}. 
    \end{align}
    The $\Phi e^{Tf}$ means multiplying a form $\eta$ by the function $e^{Tf}$, and then applying $\Phi$ to $e^{Tf}\eta$. Same for other similar notations here. The left hand side of (\ref{perturbed chain map verification process}) equals
    \begin{align}\label{lhs of checking chain map}
         \begin{bmatrix}
            \partial\Phi  & S^{-1}\partial K(\omega) + S^{-1}\cupc(\omega)\Phi \\
            0 & (-1)^{\ell-1}S^{-1}\partial\Phi 
        \end{bmatrix}\begin{bmatrix}
            e^{Tf} & \\
            & e^{Tf}
        \end{bmatrix},
    \end{align}
    and the right hand side of (\ref{perturbed chain map verification process}) equals
    \begin{align}\label{rhs of checking chain map}
        \begin{bmatrix}
            \Phi d  &  S^{-1}\Phi\omega + (-1)^{\ell-1} S^{-1} K(\omega) d \\
           0  & (-1)^{\ell-1}S^{-1}\Phi d 
        \end{bmatrix}\begin{bmatrix}
            e^{Tf} & \\
            & e^{Tf}
        \end{bmatrix}.
    \end{align}
    According to (\ref{mapping cone morse chain map statement}),  
    \begin{align}
        \partial K(\omega) + \cupc(\omega)\Phi = \Phi\omega + (-1)^{\ell-1}K(\omega) d.
    \end{align}
    Therefore, (\ref{lhs of checking chain map}) = (\ref{rhs of checking chain map}), and $\Phi_{ST}^\omega$ is a chain map. 
\end{proof}

Recall (\ref{sobolev n norm expressions}) the norm $\|\cdot\|_n$ on $\bigoplus_{k = -1}^{m+\ell-1}\left(\Omega^k(M)\oplus\Omega^{k-\ell+1}(M)\right)$. 
    Let 
    \begin{align}
    \mathbb{D} = 
    \begin{bmatrix}
        d+d^*+T\hat{c}(df) & 0\\
        0 & (-1)^{\ell-1}(d+d^*+T\hat{c}(df))
    \end{bmatrix}.
    \end{align}
For any form $\alpha\in\Omega^\bullet(M)$, let 
\begin{align}
    |\alpha| = \sqrt{g(\alpha,\alpha)}\ .
\end{align}
Then $|\alpha|$ is a smooth function on $M$.  

\begin{lemma}\label{estimate of sobolev with parameters}
    For a given $n\in\mathbb{Z}_{\geqslant 0}$, there is a constant $C_3>0$ such that when $S$ and $T$ are sufficiently large,  
    \begin{align}\label{estimate of n sobolev using garding}
       \|\mathbf{w}\|_n\leqslant C_3 T^n\left(\|(zI-\mathbb{D}_{ST})\mathbf{w}\|_{n-1}+\|\mathbf{w}\|\right)
    \end{align}
    for all $\mathbf{w}\in\bigoplus_{k = -1}^{m+\ell-1}(\Omega^k(M)\oplus\Omega^{k-\ell+1}(M))$ and all $z\in\mathbb{S}^1$.
\end{lemma}
\begin{proof}
    We write $\mathbf{w} = \begin{bmatrix}
        \alpha\\
        \beta
    \end{bmatrix}$ and then get
    \begin{align}
        \mathbb{D}_{ST}\mathbf{w} =\ & \begin{bmatrix}
            D_T\alpha + S^{-1}\omega\wedge\beta\\
            (-1)^{\ell-1}D_T\beta + S^{-1}\omega^*\lrcorner\alpha
        \end{bmatrix}, \\
        \mathbb{D}\mathbf{w} =\ & \begin{bmatrix}
        (d+d^*)\alpha\\
        (-1)^{\ell-1}(d+d^*)\beta
    \end{bmatrix}. 
    \end{align}
    Since $d+d^*$ is a first order differential operator on $\Omega^\bullet(M)$, by \cite[6.29]{warner2013foundations}, we have a 
    constant $C_4>0$ such that  
    \begin{align}
       \|\alpha\|_n\leqslant\ & C_4(\|(d+d^*)\alpha\|_{n-1}+\|\alpha\|),\\
       \|\beta\|_n\leqslant\ &  C_4(\|(-1)^{\ell-1}(d+d^*)\beta\|_{n-1}+\|\beta\|).
    \end{align}
    Therefore, we have 
    \begin{align}
        \|\mathbf{w}\|_n\leqslant \sqrt{2}C_4\left(\|\mathbb{D}\mathbf{w}\|_{n-1}+\|\mathbf{w}\|\right).
    \end{align}
    Since $zI-\mathbb{D}_{ST}$ separates into
    \begin{align}
        zI-\mathbb{D}_{ST} = -\mathbb{D} + zI - \begin{bmatrix}
            T\hat{c}(df) & \\
            & (-1)^{\ell-1} T\hat{c}(df) 
        \end{bmatrix} - \begin{bmatrix}
                & S^{-1}\omega\\
                S^{-1}\omega^* & 
            \end{bmatrix},
    \end{align}
    we find positive constants $C_4'$, $C_4''$, $C_4'''$ such that 
    \begin{align}\label{repeat it}
        & \|\mathbf{w}\|_n   \nonumber\\
        \leqslant\ & \sqrt{2}C_4\left(\|zI-\mathbb{D}_{ST}\mathbf{w}\|_{n-1} + |z|\|\mathbf{w}\|_{n-1} + C_4'T\|\mathbf{w}\|_{n-1}+ C_4'' S^{-1}\|\mathbf{w}\|_{n-1} + \|\mathbf{w}\| \right)   \nonumber\\
        \leqslant\ & C_4'''T\left(\|zI-\mathbb{D}_{ST}\mathbf{w}\|_{n-1}+\|\mathbf{w}\|_{n-1}+\|\mathbf{w}\|\right).
    \end{align}
    We repeat the procedure inductively on terms $\|\mathbf{w}\|_{n-1}$, $\|\mathbf{w}\|_{n-2}$ $\cdots$ in (\ref{repeat it}). Then, we obtain (\ref{estimate of n sobolev using garding}) when we reach $n = 0$. 
\end{proof}

\begin{remark}
    \normalfont Up to now, we do not have a particular reason to require $S$  to be ``exponentially larger'' than $T$ as in Theorem \ref{main theorem}. We will see this requirement in the next estimate. 
\end{remark}

\begin{proposition}\label{refinement of the estimates of proj vs non-proj}
For each $p\in Z^k(df)$, we write 
\begin{align}\label{estimate for the key to the final isomorphism}
   P_{ST}\circ J_T\begin{bmatrix}
            p\\
            0
        \end{bmatrix} - J_T\begin{bmatrix}
            p\\
            0
        \end{bmatrix} = \begin{bmatrix}
            \alpha\\
            \beta
        \end{bmatrix}.
\end{align}
Then, for each $n\in\mathbb{Z}_{\geqslant 0}$, we have constants $\sigma_n>0$ and $\tau_n>0$ such that 
\begin{align}\label{sobolev n norm estimate with respect to n}
    \|\alpha\|_n < \sigma_ne^{-\sigma_nT} \ \ \text{and }\ \ \|\beta\|_n< \sigma_ne^{-\sigma_nT}
\end{align}
when $S$ and $T$ satisfies $S>e^{T}>e^{\tau_n}$. Furthermore, we have constants $C_5>0$, $C_6>0$, and $T_3>0$ such that 
\begin{align}\label{corollary by warner's book}
    |\alpha| < C_5 e^{-C_6 T} \ \ \text{,}\ \ |\beta|< C_5 e^{-C_6T}
\end{align}
when $S>e^{T}>e^{T_3}$.
\end{proposition}
\begin{proof}
    For any $p\in Z^k(df)$, 
    \begin{align}
        & P_{ST}\circ J_T\begin{bmatrix}
            p\\
            0
        \end{bmatrix} - J_T\begin{bmatrix}
            p\\
            0
        \end{bmatrix}   \nonumber\\
        =\ & \dfrac{1}{2\pi\sqrt{-1}}\int_{z\in\mathbb{S}^1}\left( (z - \mathbb{D}_{ST})^{-1}\|\xi_p\|^{-1}\begin{bmatrix}
        \xi_p\\
        0
    \end{bmatrix} - z^{-1}\|\xi_p\|^{-1}\begin{bmatrix}
        \xi_p\\
        0
    \end{bmatrix} \right)dz   \nonumber\\
    =\ & \dfrac{1}{2\pi\sqrt{-1}}\int_{z\in\mathbb{S}^1}(z - \mathbb{D}_{ST})^{-1} z^{-1} \mathbb{D}_{ST} \left(\|\xi_p\|^{-1}\begin{bmatrix}
        \xi_p\\
        0
    \end{bmatrix}\right) dz. 
    \end{align}
   We find that
   \begin{align}
       & \mathbb{D}_{ST}\begin{bmatrix}
           \xi_p\\
           0
       \end{bmatrix}
       = \begin{bmatrix}
           c(d\gamma)\rho\\
           S^{-1}\omega^*\lrcorner(\gamma\rho)
       \end{bmatrix}.
   \end{align}
   By the Gaussian integral, we see that when $T$ is sufficiently large, 
   \begin{align}\label{gaussian integral}
      \dfrac{9}{10}\left(\dfrac{\pi}{T}\right)^{m/4} \leqslant \|\xi_p\|\leqslant \left(\dfrac{\pi}{T}\right)^{m/4}. 
   \end{align}
   As $\gamma$ is supported inside $U(\varepsilon/2)$ and always equal to $1$ on $U(\varepsilon/4)$, we find a constant $\sigma'_n>0$ such that 
   \begin{align}\label{first part sobolev upper bound}
       \|c(d\gamma)\rho\|_n \leqslant\ & \sigma'_n e^{-\sigma'_n T}
   \end{align}
   when $T$ is sufficiently large. For $S^{-1}\omega^*\lrcorner(\gamma\rho)$, however, we do not have $d\gamma = 0$ inside $U(\varepsilon/4)$. This is where we need the parameter $S$. Notice that
   \begin{align}
        \left|\dfrac{\partial^n}{\partial x_{j_1}\cdots \partial x_{j_n}}\exp\left(-\dfrac{T}{2}(x_1^2+\cdots+x_m^2)\right)\right| 
       \leqslant  \sigma''_n T^n
   \end{align}
   on the chart $U$. Thus, when $T$ is sufficiently large, and $S$ satisfies $S>e^{T}$, then 
   \begin{align}\label{second part sobolev upper bound}
       \|S^{-1}\omega^*\lrcorner(\gamma\rho)\|_n\leqslant \sigma_n''' T^n e^{-T}.
   \end{align}
   The numbers $\sigma_n'$, $\sigma_n''$, and $\sigma_n'''$ are constants relying on $n$. 
   By Proposition \ref{eigenvalue behavior makes some operators invertible}, we have $C_7$ and $C_7'$ such that 
   \begin{align}\label{something easy to omit, the zero th sobolev}
       \left\|(zI-\mathbb{D}_{ST})^{-1} z^{-1}\mathbb{D}_{ST} \left(\|\xi_p\|^{-1}\begin{bmatrix}
        \xi_p\\
        0
    \end{bmatrix}\right)\right\|
    \leqslant\ & C_7\left\|\mathbb{D}_{ST} \left(\|\xi_p\|^{-1}\begin{bmatrix}
        \xi_p\\
        0
    \end{bmatrix}\right)\right\|   \nonumber\\
    =\ &  C_7\|\xi_p\|^{-1}\left\|\begin{bmatrix}
           c(d\gamma)\rho\\
           S^{-1}\omega^*\lrcorner(\gamma\rho)
       \end{bmatrix}\right\|   \nonumber\\
       \leqslant\ & C_7' e^{-C_7' T} 
   \end{align}
   when $T$ is sufficiently large and $S^{-1}\omega^*\lrcorner(\gamma\rho)$. 
   By Lemma \ref{estimate of sobolev with parameters}, (\ref{first part sobolev upper bound}), (\ref{second part sobolev upper bound}), and (\ref{something easy to omit, the zero th sobolev}), we have $\tilde{\sigma}_n$ such that 
   \begin{align}
       & \left\|(zI-\mathbb{D}_{ST})^{-1} z^{-1}\mathbb{D}_{ST} \left(\|\xi_p\|^{-1}\begin{bmatrix}
        \xi_p\\
        0
    \end{bmatrix}\right)\right\|_n    \nonumber\\
    \leqslant\ &  C_3 T^n\left\{\left\|z^{-1}\mathbb{D}_{ST} \left(\|\xi_p\|^{-1}\begin{bmatrix}
        \xi_p\\
        0
    \end{bmatrix}\right)\right\|_{n-1}+\left\|(zI-\mathbb{D}_{ST})^{-1}z^{-1}\mathbb{D}_{ST} \left(\|\xi_p\|^{-1}\begin{bmatrix}
        \xi_p\\
        0
    \end{bmatrix}\right)\right\|\right\}   \nonumber\\
    \leqslant\ & \tilde{\sigma}_n e^{-\tilde{\sigma}_n T}. 
    \end{align}
We immediately obtain (\ref{sobolev n norm estimate with respect to n}). By \cite[Corollary 6.22(b)]{warner2013foundations}, we obtain (\ref{corollary by warner's book}). 
\end{proof}

\begin{remark}\label{remark for the other side that we do not repeat}
\normalfont
    The same calculations and estimates can be applied to $\begin{bmatrix}
    0\\
    q
\end{bmatrix}$ for any $q\in Z^{k-\ell+1}(df)$. In this case, we also obtain the same estimates (constants may be different) as (\ref{sobolev n norm estimate with respect to n}) and (\ref{corollary by warner's book}).
\end{remark}

The following technical lemma is for approaching the map $K(\omega)$. Let $D = d+d^*$. Then $D^2$ is the Hodge Laplacian. It has an inverse $D^{-2}$ on the orthogonal complement of the space of harmonic forms on $M$. 
\begin{lemma}\label{the second to last technical lemma}
    Suppose $\beta = \beta(S, T)$ is a smooth homogeneous form on $M$ depending on the parameters $S$ and $T$. We write $\beta$ uniquely as 
    \begin{align}
        \beta = \beta_0 + D^2\beta_1
    \end{align}
    such that $D^2\beta_0 = 0$, and $\beta_1$ is orthogonal to $\beta_0$ under the $L^2$ norm. Assume that for each $n\geqslant 0$, there is a constant $\sigma_n>0$ such that 
    \begin{align}
        \|\beta\|_n <\sigma_n e^{-\sigma_n T}
    \end{align}
    when $T$ is sufficiently large and $S>e^{T}$. Then, there are constants $C_8>0$ and $C_8'>0$ such that 
    \begin{align}\label{estimate the length of harmonic and orthogonal d*}
        |\beta_0|< C_8e^{-C_8T}\ \ \text{and}\ \ |d^*\beta_1|< C_8'e^{-C_8'T}
    \end{align}
     when $T$ is sufficiently large and $S>e^{T}$. 
\end{lemma}
\begin{proof}
    Let $n$ be a sufficiently large odd number. Let 
    \begin{align}
        \alpha_1, \cdots, \alpha_s
    \end{align}
    be an orthonormal basis (with respect to the $L^2$ norm) of the kernel of $D^2$ on $\Omega^\bullet(M)$. 
    Let $\langle\cdot,\cdot\rangle$ be the inner product associated with the $L^2$ norm. Then, 
    \begin{align}
        \|\beta_0\|_{n-1}\leqslant \ & \sum_{i = 1}^s |\langle\beta,\alpha_i\rangle|\|\alpha_i\|_{n-1}   \nonumber\\
        \leqslant\ &  \sum_{i = 1}^s \|\beta\|\cdot\|\alpha_i\|_{n-1}   \nonumber\\
        & \text{let $\text{vol}_g(M)=$ volume of $M$ under $g$}\Rightarrow   \nonumber\\
        \leqslant\ &  \sum_{i = 1}^s \sigma_0 e^{-\sigma_0 T}\sqrt{\text{vol}_g(M)}\|\alpha_i\|_{n-1}.
    \end{align}
    Since $\|\alpha_1\|_{n-1}, \cdots, \|\alpha_s\|_{n-1}$ are constants depending only on $n$, we get the first part of (\ref{estimate the length of harmonic and orthogonal d*}) by \cite[Corollary 6.22(b)]{warner2013foundations}. 

    Next, by \cite[Theorem 6.18(h)]{warner2013foundations}, we have $\sigma>0$ (depending on $n$) such that
    (To reduce the indices of constants, we absorb constants into one $\sigma$ when possible)
    \begin{align}
      \|d^*\beta_1\|_n   \leqslant\ & \sigma\|\beta_1\|_{n+1}   \nonumber\\
        \leqslant\ & \sigma (\|D^2\beta_1\|_{n-1} + \|\beta_1\|_{n-1})   \nonumber\\
        &\text{inductively on $n$}\Rightarrow    \nonumber\\
        \leqslant\ & \sigma(\|D^2\beta_1\|_{n-1} + \|\beta_1\|)    \nonumber\\
        =\ & \sigma(\|\beta-\beta_0\|_{n-1} + \|D^{-2}(\beta-\beta_0)\|)   \nonumber\\
        &\text{let $\lambda$ be the minimal nonzero eigenvalue of $D^2$ on $\Omega^\bullet(M)$}\Rightarrow    \nonumber\\
        \leqslant\ & \sigma(\|\beta-\beta_0\|_{n-1} + \lambda^{-1}\|\beta-\beta_0\|)     \nonumber\\
        \leqslant\ & \sigma e^{-\sigma T} + \lambda^{-1} \sigma e^{-\sigma T}.
    \end{align}
    Thus, use \cite[Corollary 6.22(b)]{warner2013foundations} again, we obtain
    the second half of (\ref{estimate the length of harmonic and orthogonal d*}).  
\end{proof}
Finally, we prove the cochain isomorphism for Theorem \ref{main theorem}. 
\begin{proposition}\label{final thing to check}
    There are constants $C_0>1$ and $T_0>0$ such that the map $\Phi_{ST}^\omega\circ P_{ST}\circ J_T$ is an isomorphism for all $S$ and $T$ satisfying $S>e^{C_0T}>e^{C_0T_0}$.
\end{proposition}
\begin{proof}
    For $p\in Z^k(df)$, we find 
    \begin{align}
        & \Phi_{ST}^\omega\circ P_{ST}\circ J_T \begin{bmatrix}
            p\\
            0
        \end{bmatrix}   \nonumber\\
        =\ & \Phi_{ST}^\omega\circ J_T \begin{bmatrix}
            p\\
            0
        \end{bmatrix} + \Phi_{ST}^\omega\begin{bmatrix}
            \alpha\\
            \beta
        \end{bmatrix}   \nonumber\\
        =\ & \|\xi_p\|^{-1}\Phi_{ST}^\omega \begin{bmatrix}
            \xi_p\\
            0
        \end{bmatrix} + \Phi_{ST}^\omega\begin{bmatrix}
            \alpha\\
            \beta
        \end{bmatrix}   \nonumber\\
        =\ & \|\xi_p\|^{-1} \begin{bmatrix}
            \Phi e^{Tf}\xi_p\\
            0
        \end{bmatrix} + \begin{bmatrix}
            \Phi e^{Tf}\alpha + S^{-1}K(\omega)e^{Tf}\beta\\
            S^{-1}\Phi e^{Tf} \beta
        \end{bmatrix}.
    \end{align}
    Like in \cite[Definition 6.8]{wittendeformationweipingzhang}, we let $\mathcal{F}$, $\mathcal{N}_T$, and $\mathcal{L}_S$ be the isomorphisms
    \begin{align}
    \begin{split}
        \mathcal{F}_T: C^k(f,g)\oplus C^{k-\ell+1}(f,g) &\to C^k(f,g)\oplus C^{k-\ell+1}(f,g)\\
      \text{for}\ p\in Z^k(df)\ \text{and\ } q\in Z^{k-\ell+1}(df),  \begin{bmatrix}
           p\\
           q
       \end{bmatrix} &\mapsto \begin{bmatrix}
           e^{Tf(p)}\cdot p\\
           e^{Tf(q)}\cdot q
       \end{bmatrix},
    \end{split}
    \end{align}
    \begin{align}
    \begin{split}
        \mathcal{N}_T: C^k(f,g)\oplus C^{k-\ell+1}(f,g) &\to C^k(f,g)\oplus C^{k-\ell+1}(f,g)\\
        \begin{bmatrix}
            a\\
            b
        \end{bmatrix} &\mapsto
        \begin{bmatrix}
            \left(\dfrac{\pi}{T}\right)^{\frac{k}{2}-\frac{m}{4}}\cdot a\\
           \left(\dfrac{\pi}{T}\right)^{\frac{k-\ell+1}{2}-\frac{m}{4}}\cdot b
        \end{bmatrix},
    \end{split}
    \end{align}
    and 
    \begin{align}
    \begin{split}
        \mathcal{L}_S: C^k(f,g)\oplus C^{k-\ell+1}(f,g) &\to C^k(f,g)\oplus C^{k-\ell+1}(f,g)\\
        \begin{bmatrix}
            a\\
            b
        \end{bmatrix} &\mapsto \begin{bmatrix}
            a\\
           S^{-1} b
        \end{bmatrix}.
    \end{split}
    \end{align}
    We write $\Phi e^{Tf}\xi_p$ into 
    \begin{align}
   \|\xi_p\|^{-1} \Phi e^{Tf}\xi_p = \sum_{p'\in Z^k(df)} \left(\|\xi_p\|^{-1} e^{Tf(p')} \int_{\overline{\mathcal{U}(p')}} e^{T(f-f(p'))}\xi_p\right)p'.
    \end{align}
    Since the integrals are on unstable manifolds, there must be
    \begin{align}\label{on the unstable manifold less than or equal to 0}
        f|_{\mathcal{U}(p)}-f(p')\leqslant 0, \ \ \forall\ p'\in Z^k(df). 
    \end{align}
    When $p' = p$, by (\ref{gaussian integral}), (\ref{The normal distribution}), (\ref{morse lemma coordinate expression of f}), we see that when $T$ is sufficiently large, 
    \begin{align}
       e^{Tf(p)}\dfrac{9}{10}\left(\dfrac{\pi}{T}\right)^{\frac{k}{2}-\frac{m}{4}}\leqslant  \left|\|\xi_p\|^{-1} e^{Tf(p)} \int_{\overline{\mathcal{U}(p)}} e^{T(f-f(p))}\xi_p\right| \leqslant e^{Tf(p)}\left(\dfrac{\pi}{T}\right)^{\frac{k}{2}-\frac{m}{4}}.
    \end{align}
    When $p'\neq p$, according to \cite[Theorem 2.1]{hutchingslecturenotes}, $p\notin\overline{\mathcal{U}(p')}$. By (\ref{gaussian integral}) and (\ref{on the unstable manifold less than or equal to 0}), we have $C_9>0$ such that 
    \begin{align}
        \left|\|\xi_p\|^{-1}e^{Tf(p')} \int_{\overline{\mathcal{U}(p')}}  e^{T(f-f(p'))}\xi_p\right|\leqslant e^{Tf(p')}\left(\dfrac{\pi}{T}\right)^{\frac{k}{2}-\frac{m}{4}}e^{-C_9T}
    \end{align}
    when $T$ is sufficiently large. 
    
    We now write $\Phi e^{Tf} \alpha$ and $S^{-1}\Phi e^{Tf} \beta$ as
    \begin{align}
        \Phi e^{Tf} \alpha =\ & \sum_{p'\in Z^k(df)}\left(e^{Tf(p')}\int_{\overline{\mathcal{U}(p')}} e^{T(f-f(p'))}\alpha\right)p'\ , \\
        \text{and\ }\ S^{-1}\Phi e^{Tf} \beta =\ & \sum_{q\in Z^{k-\ell+1}(df)}\left(S^{-1}e^{Tf(q)}\int_{\overline{\mathcal{U}(q)}} e^{T(f-f(q))}\beta\right)q.
    \end{align}
    By (\ref{on the unstable manifold less than or equal to 0}) and (\ref{corollary by warner's book}), we have $C_{9}'>0$ and $C_9''>0$ such that for all $1\leqslant i\leqslant r$, 
    \begin{align}
        \left|e^{Tf(p')}\int_{\overline{\mathcal{U}(p')}} e^{T(f-f(p'))}\alpha\right|\leqslant\ & e^{Tf(p')} C_9' e^{-C_9' T},\\
       \text{and\ } \left|S^{-1}e^{Tf(q)}\int_{\overline{\mathcal{U}(q)}} e^{T(f-f(q))}\beta\right|\leqslant\ & S^{-1} e^{Tf(q)} C_9'' e^{-C_9'' T}.
    \end{align}
    
    The above estimates are close to those in \cite[Theorem 6.9]{wittendeformationweipingzhang}. We need $S$ in the estimate of $S^{-1}K(\omega)e^{Tf}\beta$. Let 
    \begin{align}\label{choosing the large C0}
        y_0 = \max_{x\in M} |f(x)|,\ \text{and}\ C_0 = 1 + 2y_0.
    \end{align}
    That $1$ in $C_0$ is to fit the condition of Proposition \ref{refinement of the estimates of proj vs non-proj}.  
    When $S>e^{C_0T}$, and $T$ is sufficiently large, we estimate $S^{-1}K(\omega)e^{Tf}\beta$. 
    
    By repeating the Leibniz rule \cite[Proposition 4.15(iii)]{leeriemannian}
    of the connection $\nabla$, we get 
    \begin{align}
        \nabla(e^{Tf}\beta) =\ & e^{Tf}Tdf\otimes\beta + e^{Tf}\nabla\beta,
    \end{align}
    and 
    \begin{align}
        \nabla^2(e^{Tf}\beta) =\ & e^{Tf}T^2 df\otimes df\otimes \beta + e^{Tf}T(\nabla df)\otimes\beta \nonumber \\
       &  + e^{Tf}Tdf\otimes\nabla\beta + e^{Tf}Tdf\otimes\nabla\beta + e^{Tf}\nabla^2\beta,
    \end{align}
    and continue to $\nabla^n(e^{Tf}\beta)$ for each $n\geqslant 0$. 
    Since the manifold $M$ is compact, then like \cite[Theorem 6.18(i)]{warner2013foundations}, there is a constant $\widetilde{\zeta}_n>0$ such that  
    \begin{align}
        \|\nabla^n(e^{Tf}\beta)\|\leqslant \widetilde{\zeta}_n  e^{Ty_0} T^n\|\beta\|_n. 
    \end{align}
    Thus, we have a constant $\zeta_n>0$ such that 
    \begin{align}
        \|e^{Tf}\beta\|_n \leqslant \zeta_n  e^{Ty_0} T^n\|\beta\|_n.
    \end{align}
    Now, with (\ref{choosing the large C0}), when $S>e^{C_0T}$,  and when $T$ is sufficiently large, we find 
    \begin{align}\label{estimate when with etf and for K}
        \|S^{-1}e^{Tf}\beta\|_n\leqslant \zeta_n  e^{-T(1+y_0)} T^n\|\beta\|_n \leqslant \zeta_n  e^{-\frac{T}{2}-Ty_0} \|\beta\|_n\leqslant \zeta_n  e^{-\frac{T}{2}-Ty_0} \sigma_n e^{-\sigma_n T}.
    \end{align}
    The last inequality is by (\ref{sobolev n norm estimate with respect to n}). 
    According to Lemma \ref{the second to last technical lemma},  and (\ref{estimate when with etf and for K}), if we write $e^{Tf}\beta$ uniquely as
    \begin{align}
        S^{-1}e^{Tf}\beta = \widetilde{\beta}_0 + D^2\widetilde{\beta}_1
    \end{align}
    such that $D^2\widetilde{\beta}_0 = 0$, and $\widetilde{\beta}_1$ is orthogonal to $\widetilde{\beta}_0$ with respect to the $L^2$ norm, then we have constants $C_{10}>0$ and $C'_{10}>0$ such that 
    \begin{align}
        |\widetilde{\beta}_0|\leqslant\ & C_{10} e^{-C_{10}T}e^{-Ty_0}, \label{with parameters the harmoic part estimate}\\
        |d^*\widetilde{\beta}_1|\leqslant\ & C_{10}' e^{-C_{10}'T} e^{-Ty_0}. \label{with parameters the orthogonal green map part estimate}
    \end{align}
    By (\ref{the map K compatible with cup product}), we see that
    \begin{align}
        & S^{-1}K(\omega)e^{Tf}\beta  \nonumber\\
        =\ & (-1)^\ell \sum_{p'\in Z^k(df)}\left(\int_{\overline{\mathcal{U}(p')}} \omega\wedge d^*\widetilde{\beta}_1 - \sum_{q\in Z^{k-\ell}(df)}\int_{\overline{\mathcal{U}(q)}} d^*\widetilde{\beta}_1 \int_{\mathcal{M}(p',q)}\omega\right)p'  \nonumber\\
        & + \delta\left[\sum_{p''\in Z^{k+1}(df)}\left(\int_{\overline{\mathcal{U}(p'')}} \omega\wedge \widetilde{\beta}_0 - \sum_{r\in Z^{k-\ell+1}(df)}\int_{\overline{\mathcal{U}(r)}} \widetilde{\beta}_0 \int_{\mathcal{M}(p'',r)}\omega\right)p''\right]. 
    \end{align}
    The linear map $\delta$ is defined with respect to the inner product 
    \ref{inner product on Ck morse} on a finite dimensional space. According to (\ref{with parameters the harmoic part estimate}) and (\ref{with parameters the orthogonal green map part estimate}), when we write $S^{-1}K(\omega)e^{Tf}$ as the linear combination 
    \begin{align}
        S^{-1}K(\omega)e^{Tf}\beta = \sum_{i = 1}^r t_i\cdot p_i 
    \end{align}
    with $p_1, p_2, \cdots, p_r\in Z^k(df)$, then there are $C_{11}>0$ and $C_{12}>0$ such that 
    \begin{align}\label{estimate small terms finally turn into coefficients}
        |t_i|\leqslant C_{11} e^{-C_{12} T} e^{-Ty_0}. 
    \end{align} 
    for $1\leqslant i\leqslant r$. 
    
    By Remark \ref{remark for the other side that we do not repeat}, we can also apply the above steps to $\begin{bmatrix}
        0\\
        q
    \end{bmatrix}$ with $q\in Z^{k-\ell+1}(df)$. 
    The only tricky part is to estimate 
    $S^{-1}K(\omega) \left(\|\xi_q\|^{-1} e^{Tf}\xi_q\right)$. It follows the same steps as how we obtain (\ref{estimate small terms finally turn into coefficients}). The final expression is 
    \begin{align}
        S^{-1}K(\omega)\left( \|\xi_q\|^{-1} e^{Tf}\xi_q \right)= \sum_{j = 1}^s u_j\cdot q_j,
    \end{align}
    with $q_1, \cdots, q_s\in Z^{k-\ell+1}(df)$ and 
    \begin{align}
        |u_j|\leqslant C_{13} e^{-C_{14} T} e^{-Ty_0} 
    \end{align}
    for $1\leqslant j\leqslant s$. 
    
    Summarizing all estimates, we have  
    \begin{align}\label{separation of the operator}
        \Phi_{ST}^\omega\circ P_{ST}\circ J_T = \mathcal{F}_T\circ \mathcal{N}_T\circ \mathcal{L}_S\circ \mathcal{Y}.
    \end{align}
    Here, the operator $\mathcal{Y}$ is identified with a matrix $Y$ in this way: 
    Let $\mu = |Z(df)|$. We list all the critical points of $f$ as 
      $p_1, p_2, \cdots, p_\mu.$
    Then, we let 
    \begin{align}
        \mathbf{v}_i = \begin{bmatrix}
            p_i\\
            0
        \end{bmatrix},\ \ \mathbf{w}_i = \begin{bmatrix}
            0\\
            p_i
        \end{bmatrix}.
    \end{align}
    The matrix $Y$ is given by 
    \begin{align}
        & \begin{bmatrix}
            \mathcal{Y}\mathbf{v}_1 & \cdots & \mathcal{Y}\mathbf{v}_\mu & \mathcal{Y}\mathbf{w}_1  & \cdots & \mathcal{Y}\mathbf{w}_\mu
        \end{bmatrix} \nonumber\\
        =\ & \begin{bmatrix}
            \mathbf{v}_1 & \cdots & \mathbf{v}_\mu & \mathbf{w}_1  & \cdots & \mathbf{w}_\mu
        \end{bmatrix} Y. 
    \end{align}
    Summarizing all the estimates that we have done, there are constants $C_{15}>0$ and $C_{16}>0$ such that the entries $Y_{ij}$ satisfy:
    \begin{align}
        \dfrac{9}{10}\leqslant Y_{ii}\leqslant 1,\ \ \text{and}\ \ 
        |Y_{ij}| \leqslant C_{15}e^{-C_{16}T}\ \text{for $i\neq j$}. 
    \end{align}
    By \cite[(4.16)]{greub1981linear}, $Y$ is an invertible matrix, and thus $\mathcal{Y}$ is an isomorphism. 
    Since $\mathcal{F}_T$, $\mathcal{N}_T$, and $\mathcal{L}_S$ are already 
    isomorphisms, then $\Phi_{ST}^\omega\circ P_{ST}\circ J_T$ is also an isomorphism when $T$ is sufficiently large and $S>e^{C_0 T}$. 
    
    The lower bound $T_0$ of $T$ is chosen to make $T$ sufficiently large for all inequalities to hold. The maximal value $y_0$ of $|f|$ is put into $C_0$ because we need to ensure that $\mathcal{F}_T$ and $\mathcal{N}_T$ can be correctly placed in (\ref{separation of the operator}). 
\end{proof}
Since $P_{ST}\circ J_T$ is an isomorphism when $S>e^{C_0 T}>e^{C_0T_0}$, by Proposition \ref{final thing to check}, $\Phi_{ST}^\omega$ is an isomorphism. Since $\Phi_{ST}^\omega$ is also a cochain map, we see that it gives a cochain isomorphism, which is more than the quasi-isomorphism. The proof of Theorem \ref{main theorem} is complete. 

\section{Decomposition of cohomology}\label{section of Decomposition of instanton cohomology}
In this section, we prove Corollaries \ref{morse equalities}-\ref{decomposition corollary}. The method for Corollary \ref{morse equalities} is similar to \cite[(3.2.56)]{ma2007holomorphic} and \cite[(5.16)]{wittendeformationweipingzhang}. Then, we prove Corollary \ref{reproducing inequalities here by alternating sums} in a simple way. Finally, the proof of Corollary \ref{decomposition corollary} follows a pattern similar to \cite[(2.12)]{tangtsengclausensymplecticwitten}. 

First, we prove Corollary \ref{morse equalities}. By Hodge theory, the dimension of the $k$-th cohomology group of (\ref{analytic definition}) equals $b^\omega_k$. By Theorem \ref{main theorem}, 
\begin{align}
	\dim F_{ST}^k(f,g,\omega) = \mu_k + \mu_{k-\ell+1},
\end{align}
and then
\begin{align}
	R_k =\ & \dim F_{ST}^k(f,g,\omega) - \dim\ker\left(d_{ST}^\omega: F_{ST}^k(f,g,\omega)\to F_{ST}^{k+1}(f,g,\omega)\right) \nonumber\\
	=\ & \mu_k + \mu_{k-\ell+1} - \dim\ker\left(d_{ST}^\omega: F_{ST}^k(f,g,\omega)\to F_{ST}^{k+1}(f,g,\omega)\right) \nonumber\\
	=\ & \mu_k + \mu_{k-\ell+1} - b_k^\omega - R_{k-1}.
\end{align}
Iterate $R_k$ over $k$: 
\begin{align}
	R_k =\ & \mu_k + \mu_{k-\ell+1} - b_k^\omega - \mu_{k-1} - \mu_{k-\ell} + b_{k-1}^\omega + R_{k-2}\nonumber\\
	    =\ & \cdots
\end{align}
until $k = -1$. Since $R_{-1} = 0$, we immediately find 
\begin{align}
		R_k+\sum_{j = -1}^k (-1)^{k-j} b_j^\omega = \sum_{j = -1}^k (-1)^{k-j} (\mu_j + \mu_{j-\ell+1})
	\end{align}
	for $-1\leqslant k\leqslant m+\ell-1$. Corollary \ref{morse equalities} is proved. 
	
Second, we prove Corollary \ref{reproducing inequalities here by alternating sums}. By Theorem \ref{main theorem}, since 
\begin{align}
	\Phi_{ST}^\omega: F_{ST}^k(f,g,\omega)\to C^k(f,g)\oplus C^{k-\ell+1}(f,g)
\end{align}
is a quasi-isomorphism for $S>e^{C_0T}>e^{C_0T_0}$, we see that $R_k$ equals
{\small\begin{align}
\text{rank}\left(\begin{bmatrix}
 \partial & \cupc(\omega)\\
 0 & (-1)^{\ell-1}\partial	
 \end{bmatrix}
: C^k(f,g)\oplus C^{k-\ell+1}(f,g) \to C^{k+1}(f,g)\oplus C^{k-\ell+2}(f,g)
 \right). 
\end{align}
}

\noindent By finite dimensional linear algebra, we have 
\begin{align}
	R_k \geqslant \text{rank}(\cupc(\omega): C^{k-\ell+1}(f,g)\to C^{k+1}(f,g)) = v_{k-\ell+1}. 
\end{align}
Thus, by Corollary \ref{morse equalities}, 
   \begin{align}\label{proving the ineq with v ranks}
   	\sum_{j = -1}^k (-1)^{k-j} b_j^\omega 
   	 =\ & -R_k + \sum_{j = -1}^k (-1)^{k-j} (\mu_j + \mu_{j-\ell+1})\nonumber\\
   	\leqslant\ & -v_{k-\ell+1} + \sum_{j = -1}^k (-1)^{k-j} (\mu_j + \mu_{j-\ell+1}) \nonumber\\
   	=\ & \sum_{j = -1}^k (-1)^{k-j} (\mu_j - v_{j-\ell} + \mu_{j-\ell+1} - v_{j-\ell+1}).
   \end{align}
   The last line of (\ref{proving the ineq with v ranks}) is by adding and subtracting $v_{j-\ell}$ for $-1\leqslant j\leqslant k$. Corollary \ref{reproducing inequalities here by alternating sums} is proved. 

Third, we prove Corollary \ref{decomposition corollary}. 
We consider the following three cochain complexes $(k = -1, 0, 1, \cdots, m+\ell-1)$: 
\begin{align}
    \partial:\ & C^k(f,g)\to C^{k+1}(f,g),\\
    d^\omega_{ST}:\ & F_{ST}^k(f,g,\omega)\to F_{ST}^{k+1}(f,g,\omega),\\
    (-1)^{\ell-1}\partial:\ & C^{k}(f,g)\to C^{k+1}(f,g).
\end{align}
To simplify notations, their $k$-th cohomology groups are denoted by $H^k(C)$, $H^k(F)$, and $H^k(C)$ respectively. The $(-1)^{\ell-1}$ does not change the cohomology.
For each $k$, we let 
\begin{align}
    \iota:\ & C^k(f,g)\to C^k(f,g)\oplus C^{k-\ell+1}(f,g),\\
    \pi:\ & C^k(f,g)\oplus C^{k-\ell+1}(f,g)\to C^{k-\ell+1}(f,g).
\end{align}
Then, when $S$ and $T$ are both sufficiently large, by the isomorphism $\Phi_{ST}^\omega$, we can directly check the following short exact sequence of cochain complexes: 
\begin{equation}\label{commutative diagram between cochain complexes}
\begin{tikzcd}[row sep=large, column sep = large, every label/.append style={font=\small}]
 & 0  \arrow[d, ""'] & 0 \arrow[d, ""] & \\
\cdots \arrow[r, ""]  & C^k(f,g) \arrow[r, "\partial"] \arrow[d, "(\Phi_{ST}^\omega)^{-1}\circ\iota"'] & C^{k+1}(f,g) \arrow[r, ""] \arrow[d, "(\Phi_{ST}^\omega)^{-1}\circ\iota"] & \cdots \\
\cdots \arrow[r, ""]  & F_{ST}^k(f,g,\omega) \arrow[r, "d^\omega_{ST}"] \arrow[d, "\pi\circ\Phi_{ST}^\omega"'] & F_{ST}^{k+1}(f,g,\omega) \arrow[r, ""] \arrow[d, "\pi\circ\Phi_{ST}^\omega"] & \cdots  \\
\cdots \arrow[r, ""]  & C^{k-\ell+1}(f,g) \arrow[r, "(-1)^{\ell-1}\partial"] \arrow[d, ""'] & C^{k-\ell+2}(f,g) \arrow[r, ""] \arrow[d, ""] & \cdots  \\
  & 0                   & 0                    & 
\end{tikzcd}.
\end{equation}
It induces a long exact sequence:  
\begin{proposition}
    The diagram {\normalfont(\ref{commutative diagram between cochain complexes})} induces a long exact sequence
    \begin{equation}\label{long exact sequence confirm the cup product}
        \begin{tikzcd}[column sep=large]
\cdots \to H^{k-\ell}(C) \arrow[r, "\cupc(\omega)"] & H^k(C) \arrow[r, "(\Phi_{ST}^\omega)^{-1}\circ\iota"] & H^k(F) \arrow[r, "\pi\circ\Phi_{ST}^\omega"] & H^{k-\ell+1}(C) \arrow[r, "\cupc(\omega)"] & H^{k+1}(C) \to \cdots 
\end{tikzcd}
    \end{equation}
    under the condition of Theorem \ref{main theorem}. 
\end{proposition}
\begin{proof}
    We only need to check that the connecting map is 
    \begin{align}
        \mathcal{C}(\omega): H^{k-\ell+1}(C)\to H^{k+1}(C). 
    \end{align}
    We do diagram chasing \cite[Theorem 2.16]{allenhatcher2002}. First, suppose that $a\in C^{k-\ell+1}(f,g)$ satisfies 
    $\partial a = 0$. Then, we have
    \begin{align}
        \pi\circ\Phi^\omega_{ST}\left((\Phi^\omega_{ST})^{-1}\begin{bmatrix}
            0\\
            a
        \end{bmatrix}\right) = a. 
    \end{align}
    Next, we see that 
    \begin{align}
        d_{ST}^\omega\left((\Phi^\omega_{ST})^{-1}\begin{bmatrix}
            0\\
            a
        \end{bmatrix}\right) = (\Phi^\omega_{ST})^{-1}\partial^\omega\begin{bmatrix}
            0\\
            a
        \end{bmatrix} 
        =  (\Phi^\omega_{ST})^{-1}\begin{bmatrix}
            \cupc(\omega)a\\
            0
        \end{bmatrix}  
        =  (\Phi^\omega_{ST})^{-1}\circ\iota(\cupc(\omega)a).
    \end{align}
    Thus, the diagram chasing leads to $\cupc(\omega)a$. 
\end{proof}

By the long exact sequence (\ref{long exact sequence confirm the cup product}), we see that $H^k(F)$ is isomorphic to
\begin{align}
\begin{split}
\coker\left(\cupc(\omega): H^{k-\ell}(C)\to H^{k}(C)\right)\oplus
    \ker\left(\cupc(\omega): H^{k-\ell+1}(C)\to H^{k+1}(C)\right).
\end{split}
\end{align}
Corollary \ref{decomposition corollary} is proved.

\section*{Acknowledgments}
   I thank Prof. Erkao Bao, Prof. Xiaobo Liu, Prof. Xiang Tang, Prof. Li-Sheng Tseng, Prof. Shanwen Wang, Dr. David Clausen, Dr. Shengzhen Ning, and Dr. Junrong Yan for helpful discussions. Also, I thank Beijing International Center for Mathematical Research for providing the nice working environment. 

\section*{Conflict of interest}
There is no conflict of interest.

\section*{Data availability}
The paper has no associated data. 

\bibliographystyle{abbrv}
\bibliography{mybib.bib}
\end{document}